\pdfoutput=1
\documentclass[reqno]{amsart}

\usepackage[utf8]{inputenc}
\usepackage[usenames,dvipsnames]{color}
\usepackage{amsmath}
\usepackage{enumerate}
\usepackage{amssymb}
\usepackage{amsthm}
\usepackage{mathtools}
\usepackage[mathscr]{euscript}
\usepackage{tikz-cd}
\usepackage[hidelinks]{hyperref}
\usepackage{cleveref}

\usepackage{etoolbox}

\numberwithin{equation}{section}
\newtheorem{thm}[equation]{Theorem}
\newtheorem*{thm*}{Theorem}
\newtheorem*{prop*}{Proposition}
\newtheorem*{cor*}{Corollary}
\newtheorem{prop}[equation]{Proposition}
\newtheorem{lem}[equation]{Lemma}
\newtheorem{cor}[equation]{Corollary}

\theoremstyle{remark}
\newtheorem{defn}[equation]{Definition}
\newtheorem*{defn*}{Definition}
\newtheorem*{setup*}{Setup}

\newtheorem{ex}[equation]{Example}
\newtheorem{cons}[equation]{Construction}
\newtheorem{rem}[equation]{Remark}

\newtheorem{conv}[equation]{Convention}
\newtheorem{rec}[equation]{Recollection}
\newtheorem{hyp}[equation]{Hypothesis}
\newtheorem*{hyp*}{Hypothesis}

\newtheorem*{Ack}{Acknowledgements}
\newtheorem*{warn}{Warning}

\newcommand{\nc}{\newcommand}
\nc{\dmo}{\DeclareMathOperator}

\nc{\sfA}{\mathsf{A}}
\nc{\sfB}{\mathsf{B}}
\nc{\sfC}{\mathsf{C}}
\nc{\sfD}{\mathsf{D}}
\nc{\sfE}{\mathsf{E}}
\nc{\sfF}{\mathsf{F}}
\nc{\sfG}{\mathsf{G}}
\nc{\sfH}{\mathsf{H}}
\nc{\sfI}{\mathsf{I}}
\nc{\sfJ}{\mathsf{J}}
\nc{\sfK}{\mathsf{K}}
\nc{\sfL}{\mathsf{L}}
\nc{\sfM}{\mathsf{M}}
\nc{\sfN}{\mathsf{N}}
\nc{\sfO}{\mathsf{O}}
\nc{\sfP}{\mathsf{P}}
\nc{\sfQ}{\mathsf{Q}}
\nc{\sfR}{\mathsf{R}}
\nc{\sfS}{\mathsf{S}}
\nc{\sfT}{\mathsf{T}}
\nc{\sfU}{\mathsf{U}}
\nc{\sfV}{\mathsf{V}}
\nc{\sfW}{\mathsf{W}}
\nc{\sfX}{\mathsf{X}}
\nc{\sfY}{\mathsf{Y}}
\nc{\sfZ}{\mathsf{Z}}

\nc{\scA}{\mathscr{A}}
\nc{\scB}{\mathscr{B}}
\nc{\scC}{\mathscr{C}}
\nc{\scD}{\mathscr{D}}
\nc{\scE}{\mathscr{E}}
\nc{\scF}{\mathscr{F}}
\nc{\scG}{\mathscr{G}}
\nc{\scH}{\mathscr{H}}
\nc{\scI}{\mathscr{I}}
\nc{\scJ}{\mathscr{J}}
\nc{\scK}{\mathscr{K}}
\nc{\scL}{\mathscr{L}}
\nc{\scM}{\mathscr{M}}
\nc{\scN}{\mathscr{N}}
\nc{\scO}{\mathscr{O}}
\nc{\scP}{\mathscr{P}}
\nc{\scQ}{\mathscr{Q}}
\nc{\scR}{\mathscr{R}}
\nc{\scS}{\mathscr{S}}
\nc{\scT}{\mathscr{T}}
\nc{\scU}{\mathscr{U}}
\nc{\scV}{\mathscr{V}}
\nc{\scW}{\mathscr{W}}
\nc{\scX}{\mathscr{X}}
\nc{\scY}{\mathscr{Y}}
\nc{\scZ}{\mathscr{Z}}

\nc{\mcA}{\mathcal{A}}
\nc{\mcB}{\mathcal{B}}
\nc{\mcC}{\mathcal{C}}
\nc{\mcD}{\mathcal{D}}
\nc{\mcE}{\mathcal{E}}
\nc{\mcF}{\mathcal{F}}
\nc{\mcG}{\mathcal{G}}
\nc{\mcH}{\mathcal{H}}
\nc{\mcI}{\mathcal{I}}
\nc{\mcJ}{\mathcal{J}}
\nc{\mcK}{\mathcal{K}}
\nc{\mcL}{\mathcal{L}}
\nc{\mcM}{\mathcal{M}}
\nc{\mcN}{\mathcal{N}}
\nc{\mcO}{\mathcal{O}}
\nc{\mcP}{\mathcal{P}}
\nc{\mcQ}{\mathcal{Q}}
\nc{\mcR}{\mathcal{R}}
\nc{\mcS}{\mathcal{S}}
\nc{\mcT}{\mathcal{T}}
\nc{\mcU}{\mathcal{U}}
\nc{\mcV}{\mathcal{V}}
\nc{\mcW}{\mathcal{W}}
\nc{\mcX}{\mathcal{X}}
\nc{\mcY}{\mathcal{Y}}
\nc{\mcZ}{\mathcal{Z}}

\nc{\mfp}{\mathfrak{p}}
\nc{\mfq}{\mathfrak{q}}
\nc{\mfm}{\mathfrak{m}}
\nc{\mfj}{\mathfrak{j}}

\nc{\rmh}{\mathrm{h}}
\nc{\rmfp}{\mathrm{fp}}
\nc{\rmc}{\mathrm{c}}
\nc{\rms}{\mathrm{s}}
\nc{\rma}{\mathrm{a}}
\nc{\rmb}{\mathrm{b}}

\nc{\bfA}{\mathbf{A}}
\nc{\bfB}{\mathbf{B}}
\nc{\bfC}{\mathbf{C}}
\nc{\bfD}{\mathbf{D}}
\nc{\bfE}{\mathbf{E}}
\nc{\bfF}{\mathbf{F}}
\nc{\bfG}{\mathbf{G}}
\nc{\bfH}{\mathbf{H}}
\nc{\bfI}{\mathbf{I}}
\nc{\bfJ}{\mathbf{J}}
\nc{\bfK}{\mathbf{K}}
\nc{\bfL}{\mathbf{L}}
\nc{\bfM}{\mathbf{M}}
\nc{\bfN}{\mathbf{N}}
\nc{\bfO}{\mathbf{O}}
\nc{\bfP}{\mathbf{P}}
\nc{\bfQ}{\mathbf{Q}}
\nc{\bfR}{\mathbf{R}}
\nc{\bfS}{\mathbf{S}}
\nc{\bfT}{\mathbf{T}}
\nc{\bfU}{\mathbf{U}}
\nc{\bfV}{\mathbf{V}}
\nc{\bfW}{\mathbf{W}}
\nc{\bfX}{\mathbf{X}}
\nc{\bfY}{\mathbf{Y}}
\nc{\bfZ}{\mathbf{Z}}

\nc{\gG}{\Gamma}
\nc{\gL}{\Lambda}
\nc{\gD}{\Delta}
\nc{\gS}{\Sigma}

\nc{\ga}{\alpha}
\nc{\gb}{\beta}
\nc{\g}{\gamma}
\nc{\gd}{\delta}
\nc{\e}{\epsilon}
\nc{\gz}{\zeta}
\nc{\gh}{\eta}
\nc{\gu}{\theta}
\nc{\gi}{\iota}
\nc{\gk}{\kappa}
\nc{\gl}{\lambda}
\nc{\gm}{\mu}
\nc{\gn}{\nu}
\nc{\gj}{\xi}
\nc{\gp}{\pi}
\nc{\gr}{\rho}
\nc{\gs}{\sigma}
\nc{\gt}{\tau}
\nc{\gf}{\phi}
\nc{\gx}{\chi}
\nc{\gc}{\psi}
\nc{\go}{\omega}

\nc{\wh}{\widehat}
\nc{\wt}{\widetilde}
\nc{\ol}{\overline}
\nc{\ul}{\underline}
\nc{\tl}{\tilde}
\nc{\ot}{\otimes}
\nc{\xr}{\xrightarrow}

\nc{\ie}{\sl i.e.,}

\dmo{\Loc}{\mathsf{Loc}^\ot}
\dmo{\loc}{\mathsf{loc}^\ot}
\dmo{\sloc}{\mathsf{loc}}
\nc{\lvlloc}[1]{\mathsf{loc}_{#1}^\ot}
\dmo{\Thick}{\mathsf{Thick}^\ot}
\dmo{\thick}{\mathsf{thick}^\ot}
\dmo{\Serre}{\mathsf{Serre}^\ot}
\dmo{\Thom}{\mathsf{Thom}}

\nc{\Modc}{\mathsf{Mod}(\scT^\rmc)}
\nc{\smodc}{\mathsf{mod}(\scT^\rmc)}
\nc{\tcop}{{\{\scT^\rmc\}}^{\mathrm{op}}}

\dmo{\Hocolim}{\mathsf{hocolim}}
\dmo{\cone}{\mathsf{cone}}
\dmo{\lvl}{\mathsf{level}}

\dmo{\Spec}{\mathrm{Spec}}
\dmo{\Spc}{\mathsf{Spc}(\scT^\rmc)}
\dmo{\Spcs}{\mathsf{Spc}^\rms(\scT)}
\nc{\Spcsa}[1]{\mathsf{Spc}^\rms(#1)}
\nc{\Spcsl}[1]{\mathsf{Spc}^\rms(\scT/#1)}
\dmo{\Spch}{\mathsf{Spc}^\rmh(\scT^\rmc)}
\dmo{\SPC}{\mathsf{SPC}(\scT)}
\dmo{\Supps}{\mathrm{Supp}^\rms}
\dmo{\Supph}{\mathrm{Supp}^\rmh}
\dmo{\SUPP}{\mathrm{SUPP}}
\dmo{\Supp}{\mathrm{Supp}}
\dmo{\supps}{\mathrm{supp}^\rms}
\dmo{\supph}{\mathrm{supp}^\rmh}
\dmo{\supp}{\mathrm{supp}}

\nc{\Ab}{\mathsf{Ab}}
\nc{\Sets}{\mathsf{Set}}
\nc{\CRing}{\mathsf{CRing}}
\nc{\Top}{\mathsf{Top}}
\nc{\Ring}{\mathsf{Ring}}

\dmo{\Mod}{\mathsf{Mod}}
\dmo{\smod}{\mathsf{mod}}
\dmo{\Modu}{\ul{\Mod}}
\dmo{\smodu}{\ul{\smod}}
\dmo{\Proj}{\mathsf{Proj}}
\dmo{\Flat}{\mathsf{Flat}}
\dmo{\Inj}{\mathsf{Inj}}
\dmo{\PInj}{\mathsf{PInj}}

\dmo{\Sch}{\mathsf{Sch}}
\dmo{\Cohs}{\mathsf{Coh}}
\dmo{\QCoh}{\mathsf{QCoh}}

\dmo{\Hom}{\mathsf{Hom}}
\dmo{\shom}{\mathsf{hom}}
\dmo{\Ext}{\mathsf{Ext}}
\dmo{\Tor}{\mathsf{Tor}}
\dmo{\End}{\mathsf{End}}
\dmo{\Aut}{\mathsf{Aut}}
\dmo{\Ob}{\mathsf{Ob}}
\dmo{\Mor}{\mathsf{Mor}}
\dmo{\Ph}{\mathsf{Ph}}
\dmo{\Ann}{Ann}

\dmo{\Ker}{\mathsf{Ker}}
\dmo{\coker}{\mathsf{Coker}}
\dmo{\im}{\mathsf{Im}}
\dmo{\colim}{\mathsf{colim}}
\dmo{\Id}{Id}
\dmo{\card}{\mathsf{card}}

\nc{\Char}[1]{{\color{ForestGreen}#1}}

\tikzcdset{scale cd/.style={every label/.append style={scale=#1},
    cells={nodes={scale=#1}}}}
    
\tikzset{rot270/.style={anchor=south, rotate=270, inner sep=1.0mm}}

\nc{\rcolon}{\nobreak \mskip 6muplus1mu\mathpunct {}\nonscript \mkern -\thinmuskip {:}\mskip 2mu\relax} 

\nc{\set}[2]{\big\{\, #1 \ \big| \ #2 \,\big\}} 

\nc{\paren}[2]{\! \big(\, #1 \ \big| \ #2 \big)} 
\nc{\parens}[2]{\! \big(\, #1 \ \big| \ #2 \,\big)} 

\nc{\qquadtext}[1]{\qquad\textrm{#1}\qquad} 
\nc{\customskip}{\vspace{5 pt plus 1pt minus 1pt}}

\makeatletter
\patchcmd{\@setaddresses}{\indent}{\noindent}{}{}
\patchcmd{\@setaddresses}{\indent}{\noindent}{}{}
\patchcmd{\@setaddresses}{\indent}{\noindent}{}{}
\patchcmd{\@setaddresses}{\indent}{\noindent}{}{}
\makeatother

\begin{document}
\title{Stratification and the smashing spectrum}
\author{Charalampos Verasdanis}
\address{Charalampos Verasdanis \\ School of Mathematics and Statistics \\ University of Glasgow}
\email{c.verasdanis.1@research.gla.ac.uk}

\begin{abstract}
We develop the theory of stratification for a rigidly-compactly generated tensor-triangulated category using the smashing spectrum and the small smashing support. Within the stratified context, we investigate connections between big prime ideals, objectwise-prime ideals and homological primes, and we show that the Telescope Conjecture holds if and only if the homological spectrum is $T_0$ and the homological support detects vanishing. We also reduce stratification to smashing localizations. Moreover, we study induced maps between smashing spectra and prove a descent theorem for stratification. Outside the stratified context, we prove that the Telescope Conjecture holds if and only if the smashing spectrum is $T_0$ with respect to the small topology.
\end{abstract}

\maketitle

\vskip-\baselineskip\vskip-\baselineskip

\tableofcontents

\vskip-\baselineskip\vskip-\baselineskip

\section{Introduction}
The theory of stratification emerged from the work of Neeman on derived categories of commutative noetherian rings~\cite{Neeman92}. The central theme is the classification of localizing tensor-ideals of a rigidly-compactly generated tensor-triangulated category (\emph{big tt-category}) via means of support theory.

A systematic development of the theory of stratification was undertaken by Barthel--Heard--Sanders~\cite{BarthelHeardSanders21a,BarthelHeardSanders21b}, based on work of Benson--Iyengar--Krause~\cite{BensonIyengarKrause08,BensonIyengarKrause11a,BensonIyengarKrause11b}, using the Balmer spectrum~\cite{Balmer05} and the Balmer--Favi support~\cite{BalmerFavi11}. Under the hypothesis that the Balmer spectrum is weakly noetherian, they proved various results that make stratification easier to verify in practice. Under a stronger topological assumption, generalizing work of Stevenson~\cite{Stevenson13}, they proved that stratification implies the Telescope Conjecture, i.e., every smashing ideal of the big tt-category involved is compactly generated.

Balchin--Stevenson~\cite{BalchinStevenson21}, building on the work of Krause \cite{Krause00,Krause05} and Balmer--Krause--Stevenson~\cite{BalmerKrauseStevenson20}, based on the hypothesis that the lattice of smashing ideals $\sfS(\scT)$ of a big tt-category $\scT$ is a spatial frame, studied the support theory stemming from the smashing spectrum $\Spcs$ --- the space associated with $\sfS(\scT)$ via Stone duality. Notably, there is a surjective continuous map $\psi \colon \Spcs \to \mathsf{Spc}(\scT^\rmc)^\vee$ from the smashing spectrum to the Hochster dual of the Balmer spectrum, which is a homeomorphism if and only if $\scT$ satisfies the Telescope Conjecture. The hope is that, in examples where the Telescope Conjecture fails, the smashing spectrum (with the accompanying notion of small support) can serve as a better tool than the Balmer spectrum in classifying localizing ideals. The small smashing support is constructed by assuming that the smashing spectrum is $T_D$ --- the dual notion of ``weakly noetherian''.

\begin{hyp*}
For the rest of the introduction and throughout the paper, we will assume that the frame $\sfS(\scT)$ of smashing ideals of a big tt-category $\scT$ is a spatial frame.
\end{hyp*}

The aim of this paper is to establish results concerned with stratification, using the smashing spectrum and the small smashing support (assuming, of course, that $\Spcs$ is $T_D$) in place of the Balmer spectrum and the Balmer--Favi support. To a large extent (specifically in Sections~\ref{sec:strat},~\ref{sec:local},~\ref{sec:comp}) our results are inspired by work of Barthel--Heard--Sanders~\cite{BarthelHeardSanders21a,BarthelHeardSanders21b}. In fact, under the presence of the Telescope Conjecture, the two stratification theories are equivalent and we recover many of their results. We expand on this point in detail in~\Cref{sec:comparison-BF}.

The motivating factor to our approach to stratification is that (under the assumption that $\Spc$ is generically noetherian) the Telescope Conjecture is a consequence, and thus a necessity, for $\scT$ to be stratified by the Balmer--Favi support. The proof cannot be reproduced and it is neither obvious nor is it expected that the Telescope Conjecture is a consequence of smashing stratification. In particular, it is still unclear whether the ``\emph{Universality Theorem}'' of~\cite{BarthelHeardSanders21a} can be applied. Note, however, that an example of a category that is stratified by the small smashing support and fails the Telescope Conjecture is yet to be found. One case under investigation is the derived category of a rank $1$ non-noetherian valuation domain.

What we have to offer outside of the stratified context is a reformulation of the Telescope Conjecture in terms of the \emph{small topology} on $\Spcs$; namely, the topology with basis of open subsets consisting of the smashing supports of compact objects of $\scT$. In~\Cref{prop:tc-iff-t0}, we prove that $\psi\colon \Spcs \to \mathsf{Spc}(\scT^\rmc)^\vee$ exhibits the Hochster dual of the Balmer spectrum as the Kolmogorov quotient of the smashing spectrum equipped with the small topology. Further, $\scT$ satisfies the Telescope Conjecture if and only if the smashing spectrum is $T_0$ with respect to the small topology.

In~\Cref{sec:strat}, we first prove~\Cref{prop:strat-class}, which asserts that $\scT$ is stratified by the small smashing support if and only if $\scT$ satisfies the local-to-global principle and minimality; see also~\cite[Theorem 4.1]{BarthelHeardSanders21a}. This result is crucial: we use it to deduce that if $\scT$ is stratified by the small smashing support, then a localizing ideal is a big prime if and only if it is objectwise-prime. Moreover, there is a bijective correspondence between the set of meet-prime smashing ideals and the set of big prime localizing ideals; see~\Cref{cor:obj-big-primes} and~\Cref{cor:spcs-objprimes}.

In~\Cref{sec:local}, we reduce stratification to smashing localizations and prove the following two results: First, in~\Cref{prop:minimality-everywhere}, we show that (provided that $\scT$ satisfies the local-to-global principle) $\scT$ is stratified by the small smashing support if and only if each meet-prime smashing localization of $\scT$ is stratified by the induced small smashing support; see also~\cite[Corollary 5.3]{BarthelHeardSanders21a}. We note that the local-to-global principle hypothesis cannot be dropped. Second, in~\Cref{cor:strat-closed-covers}, we show that if $\Spcs$ is covered by a family of finitely many closed subsets whose corresponding smashing localizations are stratified, then $\scT$ is stratified. If $\scT$ satisfies the local-to-global principle, then the finiteness condition can be dropped. Compare with~\cite[Corollary 5.5]{BarthelHeardSanders21a}.

In~\Cref{sec:comp}, under the hypothesis that $\scT$ is stratified by the small smashing support, we construct an injective comparison map $\xi \colon \Spch \to \Spcs$ from the homological spectrum~\cite{Balmer20a,Balmer20b} to the smashing spectrum and in~\Cref{prop:xi-surj-supph}, we prove that $\xi$ is bijective if and only if the homological support detects vanishing of objects (cf.~\cite[Proposition 3.14]{BarthelHeardSanders21b}). Moreover, in~\Cref{thm:tc} (assuming that $\scT$ is stratified by the small smashing support) we prove that $\scT$ satisfies the Telescope Conjecture if and only if the homological spectrum is $T_0$ and the homological support detects vanishing of objects.

Finally, in~\Cref{sec:image-of-SpcsF}, we study the image of the map $\mathsf{Spc}^\rms(F)\colon \mathsf{Spc}^\rms(\scU)\to \Spcs$ induced by a coproduct-preserving tensor-triangulated functor $F\colon \scT\to \scU$. Assuming that $\mathsf{Spc}^\rms(F)$ is a homeomorphism, and under a couple of generating conditions, we prove that if $\scU$ is stratified by the small smashing support, then $\scT$ is stratified by the small smashing support.

\begin{Ack}
The author would like to thank Greg Stevenson for valuable discussions and Paul Balmer and Sira Gratz for useful suggestions and comments and the anonymous referee whose comments led to considerable improvements. 
\end{Ack}

\section{Preliminaries}
\label{sec:prem}%

\subsection{Tensor-triangulated categories}
\begin{defn}
\label{defn:tt-cat}%
A \emph{tensor-triangulated category} (\emph{tt-category}) is a triple $(\scT,\ot,1)$ comprised of a triangulated category $\scT$ and a symmetric monoidal product
\[
-\ot-\colon \scT\times \scT \to \scT
\]
(here called \emph{tensor product}) with tensor-unit $1$, which is a triangulated functor in each variable. The subcategory of compact objects of $\scT$ is denoted by $\scT^\rmc$. A triangulated subcategory $\scX\subseteq \scT$ is called a \emph{tensor-triangulated subcategory} if $1\in \scX$ and $X\ot Y\in \scX,\ \forall X,Y\in \scX$.
\end{defn}

\begin{defn}
\label{defn:big-tt-cats}%
Let $\scT=(\scT,\ot,1)$ be a tensor-triangulated category with coproducts. Then $\scT$ is called \emph{rigidly-compactly generated}, henceforth a \emph{big tt-category}, if it satisfies the following conditions:
\begin{enumerate}[\rm(a)]
\item
\label{cond:cg}%
$\scT$ is compactly generated.
\item
\label{cond:compacts-tt-subcat}%
$\scT^\rmc$ is a tensor-triangulated subcategory of $\scT$.
\item
\label{cond:rigids-compacts}%
The rigid objects of $\scT$ coincide with the compact objects.
\item
\label{cond:tensor-coproduct-preserving}%
$-\ot- \colon \scT \times \scT \to\scT$ preserves coproducts in both variables.
\end{enumerate}
\end{defn}

Condition~\eqref{cond:rigids-compacts} deserves further elaboration. As a consequence of Brown representability, for every object $X\in \scT$, the functor $X\ot-$ has a right adjoint $[X,-]$. These right adjoints assemble into a bifunctor $[-,-]\colon \scT^{\mathrm{op}} \times \scT \to \scT$ called the \emph{internal-hom}. The \emph{dual} of an object $X\in \scT$, denoted by $X^\vee$, is $[X,1]$. Let $X,Y$ be two objects of $\scT$. Tracing the identity on $Y$ through the composite
\[
\scT(Y,Y)\cong \scT(1\ot Y,Y)\xr{(\e_{X,1}\ot Y)^\ast} \scT(X^\vee\ot X\ot Y,Y)\cong \scT(X^\vee \ot Y,[X,Y])
\]
where $\e_{X,1}\colon X^\vee \ot X\to 1$ is the counit of adjunction, gives rise to a natural \emph{evaluation map} $X^\vee \ot Y\to [X,Y]$. The object $X$ is called \emph{rigid} if this natural evaluation map is an isomorphism, for all $Y\in \scT$. It is in this sense that~\Cref{defn:big-tt-cats} demands that an object $X\in \scT$ is rigid if and only it is compact.

\begin{conv}
\label{conv:big-tt}%
From now on, $\scT$ will always denote a big tt-category. All subcategories are considered full and replete.
\end{conv}

\begin{defn}
\label{defn:various-ideals}%
A localizing subcategory $\scL\subseteq \scT$ is called a \emph{localizing tensor-ideal} if $X\ot Y\in \scL,\, \forall X\in \scT,\, \forall Y\in\scL$.  For simplicity, a tensor-ideal will be called an ideal. The collection of localizing ideals of $\scT$ (resp.~thick ideals of $\scT^\rmc$) will be denoted by $\Loc(\scT)$ (resp.~$\Thick(\scT^\rmc)$). For each object $X\in \scT$, the localizing subcategory (resp.~ideal) generated by $X$, meaning the smallest localizing subcategory (resp.~ideal) that contains $X$, is denoted by $\sloc(X)$ (resp.~$\loc(X)$). A non-zero localizing ideal $\scL$ is called \emph{minimal} if it contains no non-zero proper localizing ideals. A proper thick ideal $\mfp\subseteq \scT^\rmc$ is called a \emph{prime ideal} if for all $x,y\in \scT^\rmc\colon$ $x\ot y\in \mfp \Rightarrow x\in \mfp \text{ or } y\in \mfp$.
\end{defn}

\subsection{The Balmer spectrum}\label{subsec:spc}%
\begin{defn}[{\cite{Balmer05}}]
\label{defn:Balmer-spectrum}%
The \emph{Balmer spectrum} of $\scT$, denoted by $\Spc$, is the space of prime ideals of $\scT^\rmc$. The subsets $\supp(x)=\set{\mfp \in \Spc}{x\notin \mfp}$, where $x\in \scT^\rmc$, comprise a basis of closed subsets for a topology on $\Spc$. A subset $V\subseteq \Spc$ is called a \emph{Thomason subset} if $V$ is a union of closed subsets, each with quasi-compact complement. The collections of Thomason subsets and closed subsets of $\Spc$ are denoted by $\Thom(\Spc)$ and $\mcC(\Spc)$, respectively.
\end{defn}

\begin{lem}[{\cite[Lemma 2.6]{Balmer05}}]
\label{lem:supp-properties}%
The map
\[
\supp(-)\colon \Ob(\scT^\rmc) \to \mcC(\Spc), \ x\mapsto \supp(x)
\]
satisfies the following properties:
\begin{enumerate}[\rm(a)]
\item
\label{property:supp0-supp1}%
$\supp(0)=\varnothing \qquadtext{and} \supp(1)=\Spc$.
\item
\label{property:supp-coprod}%
$\supp(x\oplus y)=\supp(x) \oplus \supp(y)$.
\item
\label{property:supp-suspesion}%
$\supp(\gS x)=\supp(x)$.
\item
\label{property:supp-triangles}%
$\supp(y)\subseteq \supp(x)\cup \supp(z)$, for any triangle $x\to y\to z\to \gS x$.
\item
\label{property:supp-tpf}%
$\supp(x\ot y)=\supp(x)\cap \supp(y)$.
\end{enumerate}
\end{lem}

The Balmer spectrum parametrizes the thick ideals of $\scT^\rmc$, in the sense of the following classification theorem.
\begin{thm}[{\cite[Theorem 4.10]{Balmer05}}]
\label{thm:spc-classification}%
The map
\[
\Thom(\Spc) \to \Thick(\scT^\rmc),\ V\mapsto \set{x\in \scT^\rmc}{\supp(x)\subseteq V}
\]
is an inclusion-preserving bijection with inverse given by mapping a thick ideal $\scI$ to the Thomason subset $\bigcup_{x\in \scI}\supp(x)$. 
\end{thm}
 
\subsection{The homological spectrum}\label{subsec:spch}%
We briefly recall some standard facts about modules and the homological spectrum. The reader that wishes to learn more about the structure of the module category should consult~\cite[Appendix A]{BalmerKrauseStevenson20}, while information about the homological spectrum can be found in~\cite{Balmer20b}.

The category $\Modc$ of additive contravariant functors $\tcop \to \Ab$ is a locally coherent Grothendieck category with a right exact tensor product defined via Day convolution. Being a Grothendieck category, $\Modc$ also admits an internal-hom functor $[-,-]\colon \Modc^{\mathrm{op}}\times \Modc \to \Modc$. The subcategory of finitely presented objects of $\Modc$ is denoted by $\smodc$. The functor $h\colon \scT \to \Modc$ that sends an object $X\in \scT$ to $\wh{X}=\Hom_{\scT}(-,X)\rvert_{\scT^\rmc}$ is called the \emph{restricted Yoneda functor} and is conservative, homological, monoidal and preserves products and coproducts. A maximal Serre ideal of $\smodc$ is called a \emph{homological prime}.

\begin{cons}
\label{cons:weak-rings}%
For each homological prime $\scB$, there exists a unique maximal localizing Serre ideal $\scB'$ of $\Modc$ that contains $\scB$, constructed in the following way: Let $\loc(\scB)$ be the localizing Serre ideal of $\Modc$ generated by $\scB$. Then the Gabriel quotient $\Modc/\loc(\scB)$ remains a Grothendieck category and inherits the monoidal structure of $\Modc$ in such a way that the quotient functor $Q_\scB\colon \Modc \to \Modc/\loc(\scB)$ is monoidal. The injective envelope of the tensor-unit of $\Modc/\loc(\scB)$ is of the form $(Q_\scB \circ h)(I_\scB)$, for some pure-injective weak ring $I_\scB\in \scT$. Then $\scB'=\Ker[-,\wh{I_\scB}]$ is the unique maximal localizing Serre ideal of $\Modc$ that contains $\scB$.
\end{cons}

\begin{defn}
\label{defn:homological-spectrum}%
The \emph{homological spectrum} of $\scT$, denoted by $\Spch$, is the set of homological primes and the \emph{homological support} of an object $X\in \scT$ is the set
\[
\supph(X)=\set{\scB\in \Spch}{\wh{X}\notin \scB'}.
\]
\end{defn}

\begin{rem}
\label{rem:supph-internal-hom}%
Due to the fact that the restricted Yoneda functor is conservative, $\supph(X)=\set{\scB\in \Spch}{[X,I_\scB]\neq 0}$.
\end{rem}

\begin{lem}[{\cite[Section 4]{Balmer20b}}]
\label{lem:supph-properties}%
The map $\supph(-),\ X\mapsto \supph(X)$ satisfies the following properties:
\begin{enumerate}[\rm(a)]
\item
\label{property:supph0-supph1}%
$\supph(0)=\varnothing \qquadtext{and} \supph(1)=\Spch$.
\item
\label{property:supph-coprod}%
$\supph(\coprod_{i\in I}X_i)=\bigcup_{i\in I}\supph(X_i)$.
\item
\label{property:supph-suspesion}%
$\supph(\gS X)=\supph(X)$.
\item
\label{property:supph-triangles}%
$\supph(Y)\subseteq \supph(X)\cup \supph(Z)$, for any triangle $X\to Y\to Z\to \gS X$.
\item
\label{property:supph-tpf}%
$\supph(X\ot Y)=\supph(X)\cap \supph(Y)$.
\end{enumerate}
\end{lem}

\begin{rem}[{\cite[Remark 3.4 \& Corollary 3.9]{Balmer20a}}]
\label{rem:comparison-phi}%
The map
\[
\phi\colon \Spch \to \Spc, \quad \scB \mapsto h^{-1}(\scB)\cap \scT^\rmc
\]
is surjective and $\phi^{-1}(\supp(x))=\supph(x),\ \forall x\in \scT^\rmc$. The last relation shows that when $\Spch$ is equipped with the topology with basis of closed subsets consisting of the homological supports of compact objects, $\phi$ is continuous.
\end{rem}

\section{The smashing spectrum}
\label{sec:spcs}%
In this section, we first recall some facts concerning smashing ideals, the smashing spectrum and the associated support theory. Under the hypothesis that the frame of smashing ideals is spatial, we prove that the Hochster dual of the Balmer spectrum is the Kolmogorov quotient of the smashing spectrum, when the latter is endowed with the topology with basis consisting of the smashing supports of compact objects. We conclude that the Telescope Conjecture holds if and only if the smashing spectrum is $T_0$ with respect to the aforementioned topology.

\subsection{Smashing ideals}
\label{subsec:smashing-ideals-support}%
\begin{defn}
\label{defn:smashing-ideals}%
A localizing ideal $\scS\subseteq \scT$ is called a \emph{smashing ideal} if the quotient functor $j_\scS\colon \scT\twoheadrightarrow \scT/\scS$ admits a right adjoint $k_\scS\colon \scT/\scS \hookrightarrow \scT$ that preserves coproducts. The collection of smashing ideals of $\scT$ is denoted by $\sfS(\scT)$.
\end{defn}

\begin{rem}
The collection $\sfS(\scT)$ is a set; see~\cite[Theorem 4.9]{Krause00} and~\cite[Remark 2.16]{BalmerFavi11}.
\end{rem}

\begin{rec}
\label{rem:idempotent-triangles}%
As explained in~\cite{BalmerFavi11}, every $\scS\in \sfS(\scT)$ corresponds to a triangle $(T_\scS)\colon e_\scS\to 1\to f_\scS$, where $e_\scS$ is a left idempotent, $f_\scS$ is a right idempotent and $e_\scS\ot f_\scS=0$ (with these three conditions being equivalent). Moreover, for any object $X\in \scT$, the associated localization triangle $i_\scS r_\scS(X)\to X\to k_\scS j_\scS(X)$, where $i_\scS\colon \scS \to \scT$ is the inclusion functor and $r_\scS\colon \scT \to \scS$ is its right adjoint, is isomorphic to $e_\scS\ot X\to X\to f_\scS\ot X$, i.e., the triangle obtained by tensoring $(T_\scS)$ with $X$. More succinctly, the localization and acyclization functors corresponding to $\scS$ are given by tensoring with $f_\scS$ and $e_\scS$, respectively.
\end{rec}

\begin{rec}
\label{rem:frame-of-smashing-ideals}%
The collection $\sfS(\scT)$ of smashing ideals, ordered by inclusion, is a \emph{complete lattice}~\cite{Krause05}. That is, every collection of smashing ideals has an infimum (\emph{meet}) and a supremum (\emph{join}). The join of a collection of smashing ideals $\{\scS_i\}_{i\in I}$ is $\bigvee_{i\in I}\scS_i\coloneqq\loc(\bigcup_{i\in I}\scS_i)$ and the meet of $\{\scS_i\}_{i\in I}$, denoted by $\bigwedge_{i\in I}\scS_i$, is the join of all smashing ideals contained in $\bigcap_{i\in I}\scS_i$. The meet of finitely many smashing ideals is their intersection. On the contrary, the meet of infinitely many smashing ideals is not necessarily their intersection; see~\cite[Remark~5.12]{BalmerKrauseStevenson20}. By~\cite[Theorem~5.5]{BalmerKrauseStevenson20}, $\sfS(\scT)$ is a \emph{frame}.~This means that finite meets distribute over arbitrary joins.
\end{rec}

\begin{rem}\label{rem:spatial-corrections}%
In a previous version of their work~\cite{BalchinStevenson21}, Balchin--Stevenson claimed that $\sfS(\scT)$ is a \emph{spatial frame}, i.e., $\sfS(\scT)$ is isomorphic to the lattice of open subsets of a topological space. Subsequently, it was discovered that there is a mistake in their proposed proof. Details and counterexamples to their arguments can be found in~\cite[Appendix A]{BalchinStevenson21}. Hence, the statement ``\emph{the frame of smashing ideals of a big tt-category is a spatial frame}'' is still an open problem. Nevertheless, even if counterexamples to this statement are found, assuming that $\sfS(\scT)$ is a spatial frame and building on this hypothesis is still fruitful since there are examples, notably derived categories of valuation domains~\cite{BazzoniStovicek17}, that satisfy it.
\end{rem}

\begin{hyp}\label{hyp:spatial}%
For the rest of this paper, we assume that the frame $\sfS(\scT)$ of smashing ideals of $\scT$ is a spatial frame.
\end{hyp}

\begin{defn}[\cite{BalchinStevenson21}]
\label{defn:smashing-spectrum}%
Under~\Cref{hyp:spatial}, the space corresponding to the spatial frame $\sfS(\scT)$ via Stone duality is called the \emph{smashing spectrum} of $\scT$ and is denoted by $\Spcs$. The smashing spectrum consists of the \emph{meet-prime} smashing ideals of~$\scT$. Namely, those smashing ideals $P$ that satisfy the following property: if $\scS_1,\scS_2$ are two smashing ideals such that $\scS_1\cap \scS_2\subseteq P$, then $\scS_1\subseteq P$ or $\scS_2\subseteq P$.
\end{defn}

\begin{rem}
For details on Stone duality, see~\cite{Johnstone82}. 
\end{rem}

\begin{rec}
\label{rem:spcs-opens}%
Under~\Cref{hyp:spatial}, the lattice isomorphism $\sfS(\scT)\to \mcO(\Spcs)$, given by Stone duality, maps a smashing ideal $\scS$ to the open subset $U_\scS=\set{P\in \Spcs}{\scS\nsubseteq P}$. Accordingly, the closed subsets of $\Spcs$, being the complements of open subsets, are of the form $V_\scS=\set{P\in \Spcs}{\scS\subseteq P}$. Since $\sfS(\scT)\to \mcO(\Spcs)$ is order-preserving, $\scS\subseteq \scR$ if and only if $U_\scS\subseteq U_\scR$. The union of a collection of open subsets $\{U_{\scS_i}\}_{i\in I}$ is given by $U_{\bigvee_{i\in I}\scS_i}$. Similarly, the intersection of two (or finitely many) open subsets $U_\scS$ and $U_\scR$ is $U_{\scS\cap \scR}$. Lastly, the closure of a point $P\in \Spcs$ is $V_P$.
\end{rec}

\begin{rec}
\label{rem:smashing-ideals-coherent}%
Under~\Cref{hyp:spatial}, since the spatial frame $\sfS(\scT)$ corresponds to $\Spcs$ via Stone duality, it follows that $\Spcs$ is a sober space, i.e., every non-empty irreducible closed subset of $\Spcs$ has a unique generic point. According to~\cite[Remark 3.2.12]{BalchinStevenson21}, $\Spcs$ most likely does not have a basis of quasi-compact open subsets.
\end{rec}

\begin{rem}
A space $X$ is called \emph{spectral} if $X$ sober, quasi-compact and $X$ has a basis of quasi-compact open subsets that is closed under finite intersections. As mentioned in~\Cref{rem:smashing-ideals-coherent}, by~\cite[Remark 3.2.12]{BalchinStevenson21}, it is likely that $\Spcs$ is not a spectral space (equivalently $\sfS(\scT)$ is not a coherent frame).
\end{rem}

\subsection{Smashing supports}
For the rest of this section, we tacitly assume that~\Cref{hyp:spatial} holds.
\label{subsec:Smashing-supports}
\begin{warn}
In the sequel, $\Supps$ will denote the big smashing support and $\supps$ will denote the small smashing support. The notation used in~\cite{BalchinStevenson21} is the opposite.
\end{warn}

\begin{defn}
\label{defn:Supps}%
For every object $X\in \scT$, the subset
\[
\Supps(X)=\set{P\in \Spcs}{X\notin P}
\]
of $\Spcs$ is called the \emph{big smashing support} of $X$.
\end{defn}

\begin{lem}[{\cite[Lemma 3.2.8]{BalchinStevenson21}}]
\label{lem:Supps-properties}%
The map
\[
\Supps(-)\colon \Ob(\scT)\to \scP(\Spcs),\ X\mapsto \Supps(X)
\]
satisfies the following properties:
\begin{enumerate}[\rm(a)]
\item
\label{property:Supp0-Supp1}%
$\Supps(0)=\varnothing \qquadtext{and} \Supps(1)=\Spcs$.
\item
\label{property:Supp-coprod}%
$\Supps(\coprod_{i\in I}X_i)=\bigcup_{i\in I}\Supps(X_i)$.
\item
\label{property:Supp-suspesion}%
$\Supps(\gS X)=\Supps(X)$.
\item
\label{property:Supp-triangles}%
$\Supps(Y)\subseteq \Supps(X)\cup \Supps(Z)$, for any triangle $X\to Y\to Z\to \gS X$.
\item
\label{property:Supp-weak-tpf}%
$\Supps(X\ot Y)\subseteq \Supps(X)\cap \Supps(Y)$.
\item
\label{property:Supp-compacts}%
$\Supps(x\ot y)=\Supps(x)\cap \Supps(y),\ \forall x,y\in \scT^\rmc$.
\end{enumerate}
\end{lem}

If $x\in \scT^\rmc$, then $\supp(x)$ is a closed subset of $\Spc$. Analogously, if $X\in \scT$ and $\loc(X)\in \sfS(\scT)$, then $\Supps(X)$ is an open subset of $\Spcs$. In particular, if $x\in \scT^\rmc$, then $\loc(x)\in \sfS(\scT)$. Thus, $\Supps(x)$ is an open subset of $\Spcs$.

\begin{lem}\label{lem:support-of-compacts-is-open}
Let $X\in \scT$ and $\scR=\bigwedge\set{\scS\in \sfS(\scT)}{X\in \scS}$. Then $U_\scR \subseteq \Supps(X)$, with equality when $\loc(X)$ is a smashing ideal.
\end{lem}

\begin{proof}
Let $\scI=\bigcap\set{\scS\in \sfS(\scT)}{X\in \scS}$. Then $U_\scR=\bigcup \set{U_\scL}{\scL\in \sfS(\scT),\ \scL\subseteq \scI}$. Let $\scL$ be a smashing ideal such that $\scL\subseteq \scI$ and assume that $P\in U_\scL$, in other words $\scL\nsubseteq P$. If $X\in P$, then $\scI\subseteq P$ (since $P$ is a smashing ideal that contains $X$). As a result, $\scL\subseteq P$, which has been ruled out by assumption. So, $X\notin P$, meaning that $P\in \Supps(X)$. If $\loc(X)$ is a smashing ideal, then $\loc(X)=\scR$. Therefore, $U_\scR=U_{\loc(X)}=\Supps(X)$.
\end{proof}

Next we discuss a refinement of the big smashing support, designed to handle non-compact objects more effectively; see~\cite[Section 3.3]{BalchinStevenson21}. This refined support, called the \emph{small smashing support}, is constructed in a way similar to the Balmer--Favi support~\cite{BalmerFavi11}. In the case of the derived category of a commutative noetherian ring it recovers Foxby's small support~\cite{Foxby79}.

\begin{defn}\label{defn:locally-closed}%
Let $X$ be a topological space. A point $x\in X$ is called \emph{locally closed} if there exist an open subset $U\subseteq X$ and a closed subset $V\subseteq X$ such that $\{x\}=U\cap V$. If every point of $X$ is locally closed, then $X$ is called $T_D$.
\end{defn}

\begin{rem}
Since the $T_D$ separation axiom does not appear to be as popular as the rest of the separation axioms, e.g., $T_0$, $T_1$, $T_2$, it might be useful to provide some explanations. To this end, let $X$ be a topological space. The \emph{specialization preorder} on the points of $X$ is defined as follows: $x\leq y$ if $x\in \ol{\{y\}}$; equivalently, $\ol{\{x\}} \subseteq \ol{\{y\}}$. The \emph{downward closure} of a point $x\in X$ is $\downarrow \! x=\set{z\in X}{z\leq x}$. Evidently, $\downarrow \! x=\ol{\{x\}}$. The space $X$ is $T_D$ if $\downarrow \! x \setminus \{x\}$ is a closed subset, for every $x\in X$. One can easily check that this definition is equivalent to the one given in~\Cref{defn:locally-closed}. Any $T_1$ space is $T_D$ and any $T_D$ space is $T_0$. The original source where the $T_D$ separation axiom was studied is~\cite{AullThron62}.
\end{rem}

\begin{rem}
By~\cite[Lemma 3.3.10]{BalchinStevenson21}, if $\Spcs$ is $T_D$, then for every point $P\in \Spcs$, there exists a smashing ideal $\scS$ such that $\{P\}=U_{\scS}\cap V_P$.
\end{rem}

\begin{defn}
\label{defn:Rickard-idempotent}%
Suppose that $\Spcs$ is $T_D$ and let $P$ be a point of $\Spcs$ and $\scS$ a smashing ideal of $\scT$ such that $\{P\}=U_\scS\cap V_P$. Then the object $\gG_P=e_\scS\ot f_P$ is called the \emph{Rickard idempotent} corresponding to $P$.
\end{defn}

\begin{rec}
\label{rem:smaller-neighborhoods-Rickards}%
Let $P\in \Spcs$ be a locally closed point and consider an open subset $U_\scS$ such that $U_\scS\cap V_P=\{P\}$. Since $\scS\nsubseteq P$, it follows that $\gG_P=e_\scS \ot f_P\neq 0$. If $U_\scR$ is another open subset that contains $P$, then $U_{\scS\cap \scR}\cap V_P=\{P\}$. It then holds that $e_{\scS}\ot f_P$ and $e_{\scS\cap \scR}\ot f_P=e_{\scS}\ot e_{\scR}\ot f_P$ are isomorphic. Therefore, restricting to smaller open neighborhoods of $P$ does not alter the Rickard idempotent $\gG_P$. More generally, if $U_{\scS_1}\cap V_{P_1}=U_{\scS_2}\cap V_{P_2}$, then $e_{\scS_1}\ot f_{P_1}\cong e_{\scS_2}\ot f_{P_2}$. This shows that $\gG_P$ does not depend on the choice of open and closed subsets whose intersection is~$P$. See~\cite[Lemma~3.3.9]{BalchinStevenson21} for details.
\end{rec}

\begin{defn}
\label{defn:supps}%
Assuming that $\Spcs$ is $T_D$, the \emph{small smashing support} of an object $X\in \scT$ is
\[
\supps(X)=\set{P\in \Spcs}{\gG_P\ot X\neq 0}.
\]
\end{defn}

\begin{lem}
\label{lem:supps-properties}%
The analogous properties~\eqref{property:Supp0-Supp1}-\eqref{property:Supp-compacts} of~\Cref{lem:Supps-properties} hold for the small smashing support. If $\scS$ is a smashing ideal, then $U_\scS=\supps(e_\scS)$, $V_\scS=\supps(f_\scS)$. Further, for all $X\in \scT$, $\supps(X)\subseteq \Supps(X)$, with equality when $X\in \scT^\rmc$.
\end{lem}

\begin{proof}
The claimed properties follow from the definition of the small smashing support. For the rest, see~\cite[Lemma 3.3.11 \& Lemma 3.3.15]{BalchinStevenson21}.
\end{proof}

Recall that the \emph{Hochster dual} of a topological space $X$ is the space $X^\vee$ with open subsets the Thomason subsets of $X$.

\begin{rec}
\label{rem:comparison-of-spcs-with-spc}%
The map $\psi \colon \Spcs \to \mathsf{Spc}(\scT^\rmc)^\vee$ that sends a meet-prime smashing ideal $P$ to $P\cap \scT^\rmc$ is surjective and continuous. Furthermore, the Telescope Conjecture for smashing ideals (in the sense of~\Cref{defn:telescope}) holds if and only if $\psi$ is a homeomorphism; see~\cite[Section~5]{BalchinStevenson21}.
\end{rec}

\subsection{The small topology}\label{subsec:the-small-topology}%

The collection $\set{\Supps(x)\subseteq \Spcs}{x\in \scT^\rmc}$ is a basis for a topology on $\Spcs$, which we denote $\sfT$ and call the \emph{small topology}. The small topology is coarser than the topology provided by Stone duality (henceforth called the \emph{standard topology}).

\begin{rem}
The comparison map $\psi\colon \Spcs \to \mathsf{Spc}(\scT^\rmc)^\vee$ remains continuous when $\Spcs$ is equipped with the small topology, since for all $x\in \scT^\rmc$, $\psi^{-1}(\supp(x))=\Supps(x)$, which is open in the small topology.
\end{rem}
\begin{lem}
\label{lem:coarse-closure-of-point}%
Let $P$ be a point of $\Spcs$. The closure of $P$ in the small topology is given by $\ol{\{P\}}^\sfT=V_{\loc(P^\rmc)}$.
\end{lem}

\begin{proof}
The basic closed subsets of $\sfT$ are those of the form $V_{\loc(x)}$, where $x\in \scT^\rmc$. Thus,
\begin{align*}
\ol{\{P\}}^\sfT&=\bigcap_{P\in V_{\loc(x)}}V_{\loc(x)}\\
&=V_{\bigvee \! \big(\! \loc(x) \, \big| \, \loc(x)\subseteq P \big)}\\
&=V_{\bigvee \! \big(\! \loc(x) \, \big| \, x\in P^\rmc \big)}\\
&=V_{\loc \! \big(\! \bigcup_{x\in P^\rmc}\loc(x) \big)}\\
&=V_{\loc(P^\rmc)}.\qedhere
\end{align*}
\end{proof}

\begin{rem}
The \emph{Kolmogorov quotient} $\sfK\sfQ(X)$ of a topological space $X$ is the quotient of $X$ with respect to the equivalence relation that identifies two points $x,y\in X$ if $\ol{\{x\}}=\ol{\{y\}}$. The space $\sfK\sfQ(X)$ is $T_0$ and the quotient map $X \to \sfK\sfQ(X)$ is a surjective continuous map that is an initial object in the category of continuous maps out of $X$ into $T_0$ spaces. The space $X$ is $T_0$ if and only if the quotient map $X\to \sfK\sfQ(X)$ is a homeomorphism.
\end{rem}

\begin{defn}\label{defn:telescope}%
Let $\scT$ be a rigidly-compactly generated tensor-triangulated category. We say that $\scT$ satisfies the \emph{Telescope Conjecture} if every smashing ideal $\scS$ of~$\scT$ is \emph{compactly generated}, i.e., it holds that $\scS=\loc(\{x_i\}_{i\in I})$, where $\{x_i\}_{i\in I}$ is a collection of compact objects of $\scT$.
\end{defn}

\begin{rem}\label{rem:telescope}%
If $\scS$ is a compactly generated localizing ideal of $\scT$, then $\scS$ is a smashing ideal~\cite{Miller92}. The Telescope Conjecture is concerned with the converse statement. There are cases where it is known to be true, such as for derived categories of commutative noetherian rings~\cite{Neeman92} (or  noetherian schemes more generally~\cite{TarrioLopezSalorio04}), derived categories of absolutely flat rings~\cite{Stevenson14,BazzoniStovicek17}, and there are cases where it fails: derived categories of some non-noetherian valuation domains~\cite[Example 5.24]{BazzoniStovicek17} (see also~\cite[Section 7]{BalchinStevenson21}) and a construction by Keller~\cite{Keller94}.
\end{rem}

\begin{prop}
\label{prop:tc-iff-t0}%
The map $\psi\colon \Spcs \to \mathsf{Spc}(\scT^\rmc)^\vee$ exhibits $\mathsf{Spc}(\scT^\rmc)^\vee$ as the Kolmogorov quotient of $(\Spcs,\sfT)$. Moreover, the following are equivalent:
\begin{enumerate}[\rm(a)]
\item
\label{cond:tc-iff-spectral}%
$(\Spcs,\sfT)$ is spectral.
\item
\label{cond:tc-iff-t_0}%
$(\Spcs,\sfT)$ is $T_0$.
\item
\label{cond:tc-iff-small-homeo}%
$\psi\colon (\Spcs,\sfT)\to \mathsf{Spc}(\scT^\rmc)^\vee$ is a homeomorphism.
\item
\label{cond:tc-iff-homeo}%
$\psi\colon \Spcs \to \mathsf{Spc}(\scT^\rmc)^\vee$ is a homeomorphism.
\item
\label{cond:tc-iff-tc}%
$\scT$ satisfies the Telescope Conjecture.
\item
\label{cond:tc-iff-small=standard}%
The small and standard topologies on $\Spcs$ coincide.
\end{enumerate}
\end{prop}

\begin{proof}
Let $P,Q\in \Spcs$. By~\Cref{lem:coarse-closure-of-point}, $\ol{\{P\}}^\sfT=\ol{\{Q\}}^\sfT$ if and only if $V_{\loc(P^\rmc)}=V_{\loc(Q^\rmc)}$. It follows, by Stone duality, that $\loc(P^\rmc)=\loc(Q^\rmc)$. Consequently, $P^\rmc=Q^\rmc$, i.e., $\psi(P)=\psi(Q)$. Let $V$ be a subset of $\Spc$ such that $\psi^{-1}(V)$ is closed in the small topology. Then $\psi^{-1}(V)=\bigcap_{i\in I}V_{\loc(x_i)}$, for some family of compact objects $\{x_i\}_{i\in I}$. Since $\Spcs \setminus \psi^{-1}(\supp(x_i))=\Spcs \setminus U_{\loc(x_i)}=V_{\loc(x_i)}$, it follows that $\psi^{-1}(V)=\psi^{-1}(\bigcap_{i \in I}(\Spc \setminus \supp(x_i)))$. Since $\psi$ is surjective, $V=\bigcap_{i \in I}(\Spc \setminus \supp(x_i))$, thus $V$ is closed in $\mathsf{Spc}(\scT^\rmc)^\vee$. This shows that $\psi$ is a quotient map. Therefore, $\sfK\sfQ((\Spcs,\sfT))=\mathsf{Spc}(\scT^\rmc)^\vee$.

As an immediate consequence, $(\Spcs,\sfT)$ is $T_0$ if and only if $\psi\colon (\Spcs,\sfT)\to \mathsf{Spc}(\scT^\rmc)^\vee$ is a homeomorphism. The equivalence of~\eqref{cond:tc-iff-spectral},~\eqref{cond:tc-iff-t_0},~\eqref{cond:tc-iff-small-homeo},~\eqref{cond:tc-iff-homeo} and~\eqref{cond:tc-iff-tc} is clear. Let $\scS$ be a smashing ideal. Then $U_\scS$ is open in the small topology if and only if $U_\scS=\bigcup_{i\in I}\Supps(x_i)=U_{\loc(\coprod_{i\in I}x_i)}$, for some family $\{x_i\}_{i\in I}\subseteq \scT^\rmc$. Equivalently, $\scS=\loc(\coprod_{i\in I}x_i)=\loc(\{x_i\}_{i\in I})$, so $\scS$ is compactly generated. This shows that~\eqref{cond:tc-iff-tc}~$\Leftrightarrow$~\eqref{cond:tc-iff-small=standard}.
\end{proof}

\begin{rem}
\label{rem:t_0-homological-spectrum}%
Barthel--Heard--Sanders~\cite{BarthelHeardSanders21b} proved an analogous result for the homological spectrum: The map $\phi\colon \Spch \to \Spc$, as in~\Cref{rem:comparison-phi}, exhibits $\Spc$ as the Kolmogorov quotient of $\Spch$. Thus, $\Spch$ is $T_0$ if and only if $\phi$ is a homeomorphism. In all examples where the map $\phi$ has been computed, it is known to be a homeomorphism;  see~\cite{Balmer20a}. Whether this is always true or not is still under investigation.~In contrast, as was shown in~\Cref{prop:tc-iff-t0}, the smashing spectrum is $T_0$ with respect to the topology with basis consisting of the supports of compact objects if and only if the Telescope Conjecture holds.
\end{rem}

\begin{rem}
Recall that the small topology on $\Spcs$ is generated by the subsets $\Supps(x)$, where $x\in \scT^\rmc$.~Let $x\in\scT^\rmc$.~Since $\Supps(x)$ is quasi-compact in the standard topology on $\Spcs$, which is finer than the small topology, it follows that $\Supps(x)$ is quasi-compact in the small topology.  As discussed in~\Cref{lem:support-of-compacts-is-open}, $\Supps(x)=U_{\loc(x)}$.~Therefore, the subsets $\Supps(x)$, where $x\in \scT^\rmc$, comprise a basis of quasi-compact open subsets for the small topology on $\Spcs$.~Since $\loc(x) \cap \loc(y)=\loc(x\ot y), \ \forall x,y\in \scT^\rmc$, it follows that the small topology has a basis of quasi-compact open subsets that is closed under finite intersections. In general, the small topology is not sober; see~\cite[Section 7]{BalchinStevenson21}.
\end{rem}

\section{Stratification}
\label{sec:strat}%

Throughout this section, we assume that~\Cref{hyp:spatial} holds. The first goal of this section is to prove~\Cref{prop:strat-class}, showing that stratification by the small smashing support is equivalent to two conditions --- the local-to-global principle and minimality --- that make it easier to verify stratification in practice. If the Telescope Conjecture holds,~\Cref{prop:strat-class} recovers~\cite[Theorem 4.1]{BarthelHeardSanders21a}. The second goal is to use~\Cref{prop:strat-class} in order to establish a bijective correspondence between the smashing spectrum and the big spectrum of a stratified big tt-category.
\begin{defn}\label{defn:strat}%
A big tt-category $\scT$ such that $\Spcs$ is $T_D$ is \emph{stratified by the small smashing support} if the maps
$
\begin{tikzcd}[cramped]
\scP(\Spcs) \rar[shift left,"\tau"] & \lar[shift left,"\sigma"] \Loc(\scT),
\end{tikzcd}%
$
between the powerset of the smashing spectrum and the collection of localizing ideals of $\scT$, defined by
\begin{equation*}
\tau(W)=\set{X\in \scT}{\supps(X)\subseteq W} \qquadtext{\&} \gs(\scL)=\bigcup_{X\in \scL}\supps(X)
\end{equation*}
are mutually inverse bijections.
\end{defn}

\begin{rem}
By the properties of $\supps$, it is clear that $\tau$ and $\sigma$ are well-defined.
\end{rem}

\begin{defn}
\label{defn:ltg-min}%
Let $\scT$ be a big tt-category such that $\Spcs$ is $T_D$.
\begin{enumerate}[\rm(a)]
\item
$\scT$ satisfies the \emph{local-to-global principle} if for every object $X\in \scT$, it holds that $\loc(X)=\loc\paren{\gG_P\ot X}{P\in \Spcs}$.
\item
$\scT$ satisfies \emph{minimality} if for every $P\in \Spcs$, it holds that $\loc(\gG_P)$ is a minimal localizing tensor-ideal.
\end{enumerate}
\end{defn}

\begin{rem}
Suppose that $\scT$ satisfies the local-to-global principle. Since $\scT=\loc(1)$, it follows that $\scT=\loc\paren{\gG_P}{P\in \Spcs}$.
\end{rem}

\begin{rem}
\label{rem:ess-image}%
It is clear that if $\loc(\gG_P)$ is minimal, then $\loc(\gG_P)=\loc(\gG_P\ot X)$, for every object $X\in \scT$ such that $\gG_P\ot X\neq 0$. Provided that the local-to-global principle holds, the converse also holds.
\end{rem}

\begin{rem}
\label{rem:supps-vanishing}%
If $\scT$ satisfies the local-to-global principle, then $\supps$ detects vanishing of objects, i.e., $\supps(X)=\varnothing \Leftrightarrow X=0$. Another simple observation is that
\begin{equation}
\label{eq:loc-optimal}%
\loc\paren{\gG_P\ot X}{P\in \Spcs}=\loc \parens{\gG_P\ot X}{P\in \supps(X)}.
\end{equation}
\end{rem}

\begin{rem}
\label{rem:spcs-extend}%
If $\scS\in \sfS(\scT)$, then $\gs(\scS)=\bigcup_{X\in \scS}\supps(X)\subseteq U_\scS=\supps(e_\scS)$. Since $e_\scS\in \scS$, it follows that $\gs(\scS)=U_\scS$. In other words, the map $\sigma$ extends the bijection $\sfS(\scT) \to \mcO(\Spcs)$ that takes a smashing ideal $\scS$ to the open subset $U_\scS$. The maps $\sigma$, $\tau$, and their respective restrictions, assemble into the diagram
\begin{equation*}
\begin{tikzcd}[row sep=3.5em]
\mcO(\Spcs) \rar["\cong"',shift left=1ex] \dar[hook] & \mathmakebox[\widthof{$\Loc(\scT)$}]{\sfS(\scT)} \lar[shift left=1ex] \dar[hook]
\\
\scP(\Spcs) \rar["\tau",tail,shift left=1ex] & \Loc(\scT) \lar["\gs",two heads,shift left=1ex]
\end{tikzcd}
\end{equation*}
where the two obvious squares commute.
\end{rem}

\subsection{Classification of localizing ideals}\label{subsec:class-of-loc-ideals}%
\begin{lem}
\label{lem:tau-inj}%
It holds that $\gs \circ \gt=\Id$ (therefore, $\tau$ is injective and $\sigma$ is surjective).
\end{lem}

\begin{proof}
If $W\in \scP(\Spcs)$, then $\sigma(\tau(W))=\bigcup_{\supps(X)\subseteq W}\supps(X)\subseteq W$. In addition, $\supps(\gG_P)=\{P\}\subseteq W,\ \forall P\in W$. Thus, $W\subseteq \sigma(\tau(W))$, showing that $\sigma\circ \tau = \Id$, which proves the statement.
\end{proof}

\begin{lem}
\label{prop:min-supp}%
Suppose that $\scT$ satisfies minimality. Then the following hold:
\begin{enumerate}[\rm(a)]
\item
$\loc\paren{\gG_P\ot X}{P\in \Spcs}=\loc\parens{\gG_P}{P\in \supps(X)},\ \forall X\in \scT$.
\item
$\loc\parens{\gG_P}{P\in \sigma(\scL)}\subseteq \scL,\ \forall \scL \in \Loc(\scT)$.
\end{enumerate}
\end{lem}

\begin{proof}
$\phantom{}$
\begin{enumerate}[\rm(a)]
\item
If $P\in \supps(X)$, i.e., $\gG_P\ot X\neq 0$, then $\loc(\gG_P)=\loc(\gG_P\ot X)$ due to minimality of $\loc(\gG_P)$. In conjunction with~\eqref{eq:loc-optimal}:
\begin{equation*}
\loc\parens{\gG_P}{P\in \supps(X)}=\loc\paren{\gG_P\ot X}{P\in \Spcs}.
\end{equation*}
\item
Let $P\in \gs(\scL)$. Then there exists an object $X\in \scL$ such that $\gG_P\ot X\neq 0$. Since $X\in \scL$, it holds that $\gG_P\ot X\in \scL$. So, $\loc(\gG_P)=\loc(\gG_P\ot X)\subseteq \scL$. This proves that $\gG_P\in \scL,\ \forall P\in \sigma(\scL)$. Hence, $\loc\parens{\gG_P}{P\in \sigma(\scL)}\subseteq \scL$.\qedhere
\end{enumerate}
\end{proof}

\begin{lem}
\label{lem:supp-gen}%
Let $\scE$ be a set of objects of $\scT$. Then $\sigma(\loc(\scE))=\bigcup_{X\in \scE}\supps(X)$.
\end{lem}

\begin{proof}
The result is deduced by the following host of equivalences making use, in the second one, of the fact that $\Ker(\gG_P\ot -)$ is a localizing ideal:
\begin{align*}
P\notin \bigcup_{X\in \scE}\supps(X) &\Leftrightarrow \scE\subseteq \Ker(\gG_P\ot -) \\[-8pt] &\Leftrightarrow \loc(\scE)\subseteq \Ker(\gG_P\ot -)\\
&\Leftrightarrow P\notin \bigcup_{X\in \loc(\scE)}\supps(X)=\sigma(\loc(\scE)).\qedhere
\end{align*}
\end{proof}

\begin{rem}
\label{rem:strat-supp-ideals}%
Suppose that $\scT$ is stratified by the small smashing support. Invoking~\Cref{lem:supp-gen}, we obtain:
$
\supps(X)=\supps(Y)\Leftrightarrow \loc(X)=\loc(Y)
$
(which implies that $\supps(X)=\Spcs\Leftrightarrow \loc(X)=\scT$). Consequently, the small smashing support distinguishes between objects that generate different localizing ideals. More precisely, we can define an equivalence relation on the class of objects of $\scT$ by declaring two objects $X,Y$ to be equivalent if $\loc(X)=\loc(Y)$. In case $\scT$ is stratified by the small smashing support, $X$ is equivalent to $Y$ if and only if $\supps(X)=\supps(Y)$.
\end{rem}

\begin{lem}
\label{cor:ltg-and-gp-tensor-x}
Suppose that $\scT$ satisfies the local-to-global principle and let $X\in \scT$ be a non-zero object and $P\in \Spcs$. If $X\in \loc(\gG_P)$, then $\gG_P\ot X\neq 0$.
\end{lem}

\begin{proof}
Since $X\in \loc(\gG_P)$, it holds that $\loc(X)\subseteq \loc(\gG_P)$. \Cref{lem:supp-gen} and the fact that $\gs$ is order-preserving lead to:
\[
\supps(X)=\gs(\loc(X))\subseteq \gs(\loc(\gG_P))=\supps(\gG_P)=\{P\}.
\]
As a result, $\supps(X)$ is either empty or equal to $\{P\}$.~\Cref{rem:supps-vanishing} implies that $\supps(X)\neq \varnothing$. Consequently, $\supps(X)=\{P\}$. Hence, $\gG_P\ot X\neq 0$.
\end{proof}

\begin{rem}
If $X\in \loc(\gG_P)$ is a non-zero object and $\loc(\gG_P)$ is minimal, then $\Ker(\gG_P\ot -)\cap \loc(\gG_P)$, being a localizing ideal, is either zero or $\loc(\gG_P)$. The latter cannot hold, since $\gG_P \neq 0$. It follows that $\gG_P\ot X\neq 0$.
\end{rem}

\begin{thm}
\label{prop:strat-class}%
The category $\scT$ is stratified by the small smashing support if and only if $\scT$ satisfies the local-to-global principle and minimality.
\end{thm}

\begin{proof}
($\Rightarrow$) Suppose that $\scT$ satisfies the local-to-global principle and minimality. Since $\sigma \circ \tau=\Id$, it suffices to show that $\tau\circ \sigma=\Id$.~Let $\scL$ be a localizing ideal of $\scT$.~The relation $\scL\subseteq (\tau\circ\sigma)(\scL)$ follows from the definition of $\tau$ and $\sigma$.~Let $X\in (\tau\circ\sigma)(\scL)$, i.e., $\supps(X)\subseteq \sigma(\scL)$.~Then
\begin{align*}
\loc(X)&=\loc\paren{\gG_P\ot X}{P\in \Spcs}\tag{local-to-global principle}\\[-1.45pt]
&=\loc\parens{\gG_P\ot X}{P\in \supps(X)} \tag{\ref{eq:loc-optimal}}\\[-1.45pt]
&=\loc\parens{\gG_P}{P\in \supps(X)} \tag{\Cref{prop:min-supp}}\\[-1.45pt]
&\subseteq \loc\parens{\gG_P}{P\in \sigma(\scL)} \tag{$\supps(X)\subseteq \sigma(\scL)$}\\[-1.45pt]
&\subseteq \scL. \tag{\Cref{prop:min-supp}}
\end{align*}
As a result, $X\in \scL$. So, $(\tau\circ \sigma)(\scL)\subseteq \scL$ implying that $\tau\circ \sigma=\Id$.

($\Leftarrow$) Let $X$ be an object of $\scT$ such that $\gG_P\ot X\neq 0$, i.e., $P\in \supps(X)$. Then $\supps(\gG_P\ot X)\subseteq \supps(\gG_P)\cap \supps(X)=\{P\}\cap \supps(X)=\{P\}$. Since $\gG_P$ is an idempotent, $P\in \supps(\gG_P\ot X)$. Therefore, $\supps(\gG_P\ot X)=\{P\}=\supps(\gG_P)$. According to~\Cref{lem:supp-gen}, $\gs(\loc(\gG_P\ot X))=\supps(\gG_P\ot X)=\supps(\gG_P)=\gs(\loc(\gG_P))$. Since $\sigma$ is injective, $\loc(\gG_P\ot X)=\loc(\gG_P)$. This establishes minimality. Next, use the relation $\supps(\gG_P\ot X)=\{P\}$, when $P\in \supps(X)$, and~\Cref{lem:supp-gen} to deduce that
\begin{align*}
\sigma \big(\loc\parens{\gG_P\ot X}{P\in \supps(X)}\big)&=\bigcup_{P\in \supps(X)}\supps(\gG_P\ot X)\\
&=\bigcup_{P\in \supps(X)}\{P\}\\
&=\supps(X)\\
&=\sigma(\loc(X)).
\end{align*}
Since $\sigma$ is injective, $\loc(X)=\loc\parens{\gG_P\ot X}{P\in \supps(X)}$. Consequently, $\scT$ satisfies the local-to-global-principle.
\end{proof}

\begin{cor}
\label{cor:loc-set}%
If $\scT$ satisfies the local-to-global principle and minimality, then the collection $\Loc(\scT)$ of localizing ideals of $\scT$ is a set and every localizing ideal of $\scT$ is generated by a set of objects, hence by a single object.
\end{cor}

\begin{proof}
The first half of the statement emanates instantly from~\Cref{prop:strat-class}. The second half stems from~\cite[Lemma~3.3.1]{KrauseStevenson17} by specializing the arguments to localizing ideals instead of general localizing subcategories, as noted in~\cite[Proposition~3.5]{BarthelHeardSanders21a}.
\end{proof}

\subsection{Objectwise and big primes}\label{subsec:Obj-and-big-primes}%

\begin{defn}\label{defn:types-of-ideals}%
Let $\scL$ be a proper localizing tensor-ideal of $\scT$.
\begin{enumerate}[\rm(a)]
\item
$\scL$ is called \emph{objectwise-prime} if $X\ot Y\in \scL$ implies $X\in \scL$ or $Y\in \scL$.
\item
$\scL$ is called \emph{radical} if $X^{\ot n}\in \scL$ implies $X\in \scL$, for all $n\geq 1$.
\item
$\scL$ is called a \emph{big prime} if $\scL$ is radical and $\scI_1\cap \scI_2\subseteq \scL$ implies $\scI_1\subseteq \scL$ or $\scI_2\subseteq \scL$, for all radical localizing tensor-ideals $\scI_1,\scI_2$.
\end{enumerate}
The collection of big prime localizing tensor-ideals of $\scT$, introduced in~\cite{BalchinStevenson21} and called the \emph{big spectrum} of $\scT$, is denoted by $\SPC$.
Evidently, if $\scL$ is objectwise-prime, then $\scL$ is radical.
\end{defn}

\begin{prop}
\label{prop:tpf}%
Suppose that $\scT$ satisfies minimality. Then the following hold:
\begin{enumerate}[\rm(a)]
\item
\label{cond:tpf}%
$\supps(X\ot Y)=\supps(X)\cap \supps(Y),\ \forall X,Y\in \scT$. $($Tensor Product Formula$)$
\item
\label{cond:kernels-are-obj-primes}%
$\Ker(\gG_P\ot -)$ is an objectwise-prime localizing ideal of $\scT,\ \forall P\in \Spcs$.
\end{enumerate}
\end{prop}

\begin{proof}
First of all, the statements~\eqref{cond:tpf} and~\eqref{cond:kernels-are-obj-primes} are equivalent, since they both state: $\forall X,Y\in \scT,\ \forall P\in \Spcs$: $\gG_P\ot X\ot Y=0$ if and only if $\gG_P\ot X=0$ or $\gG_P\ot Y=0$. Let $X,Y\in \scT$ such that $\gG_P\ot X\neq 0$ and $\gG_P\ot Y\neq 0$. If $Y\in \Ker(\gG_P\ot X\ot -)$, then $\loc(\gG_P)=\loc(\gG_P\ot Y)\subseteq \loc(Y)\subseteq \Ker(\gG_P\ot X\ot -)$. Therefore, $\gG_P\in \Ker(\gG_P\ot X\ot -)$, which is a contradiction, since we assumed that $\gG_P\ot X\neq 0$. This proves that $\gG_P\ot X\ot Y\neq 0$ and the proof is complete.
\end{proof}

\begin{rem}
\label{rem:nilpotents}%
From~\Cref{rem:supps-vanishing} and~\Cref{prop:tpf}, we learn that if $\scT$ is stratified, then $\supps$ detects vanishing of objects and satisfies the Tensor Product Formula. As a consequence, $\scT$ cannot have any non-zero $\ot$-nilpotent objects.
\end{rem}

\begin{lem}
\label{lem:int-ker}%
Let $\scL$ be a localizing ideal of $\scT$. Then
\[
\tau(\sigma(\scL))=\bigcap_{\scL\subseteq \Ker(\gG_P\ot -)}\Ker(\gG_P\ot -).
\]
\end{lem}

\begin{proof}
Let $X$ be an object of $\scT$. Then $X\notin \bigcap_{\scL\subseteq \Ker(\gG_P\ot -)}\Ker(\gG_P\ot -)$ if and only if there exists $P\in \Spcs$ such that $\scL \subseteq \Ker(\gG_P\ot -)$ and $\gG_P\ot X\neq 0$. Equivalently, $P\in \supps(X)$ and $P\notin \bigcup_{Y\in\scL}\supps(Y)=\gs(\scL)$. In other words, $\supps(X)\nsubseteq \sigma(\scL)$. By definition of $\tau$, the latter happens if and only if $X\notin \tau(\sigma(\scL))$.
\end{proof}

\begin{cor}
\label{thm:radicals}%
If $\scT$ is stratified by the small smashing support, then all localizing ideals of $\scT$ are radical.
\end{cor}

\begin{proof}
Let $\scL$ be a localizing ideal of $\scT$. Since $\scT$ is stratified, it holds that $\tau\circ\sigma=\Id$. Further,~\Cref{lem:int-ker} implies that $\scL=\bigcap_{\scL\subseteq \Ker(\gG_P\ot -)}\Ker(\gG_P\ot -)$. By~\Cref{prop:strat-class}, $\scT$ satisfies minimality. According to~\Cref{prop:tpf}, each $\Ker(\gG_P\ot -)$ is objectwise-prime, hence radical. Since radical ideals are closed under intersections, $\scL$ is radical.
\end{proof}

\begin{lem}
\label{lem:form-ker}%
Suppose that $\scT$ satisfies the local-to-global principle. Then for all $P\in \Spcs$, it holds that $\Ker(\gG_P\ot -)=\loc\parens{\gG_Q}{Q\neq P}$.
\end{lem}

\begin{proof}
Let $P$ and $Q$ be distinct meet-prime smashing ideals. Since $\gG_P\ot \gG_Q=0$, it follows that $\gG_Q\in \Ker(\gG_P\ot -)$. Therefore, $\loc\parens{\gG_Q}{Q\neq P}\subseteq \Ker(\gG_P\ot -)$. If $X\in \Ker(\gG_P\ot -)$,
then $\loc(X)=\loc\parens{\gG_Q\ot X}{Q\neq P}\subseteq \loc\parens{\gG_Q}{Q\neq P}$. Hence, $\Ker(\gG_P\ot -)\subseteq \loc\parens{\gG_Q}{Q\neq  P}$, proving the claimed equality.
\end{proof}

\begin{prop}
\label{prop:objprimes}%
Suppose that $\scT$ is stratified by the small smashing support. Then every objectwise-prime localizing ideal $\scL$ is of the form $\Ker(\gG_P\ot -)$, for a unique $P\in \Spcs$.
\end{prop}

\begin{proof}
If $P$ and $Q$ are distinct meet-prime smashing ideals, then $\gG_P\ot \gG_Q=0\in \scL$. Therefore, $\gG_P\in \scL$ or $\gG_Q\in \scL$. Since $\scT=\loc\paren{\gG_P}{P\in \Spcs}$, and $\scL$ is proper by definition, $\scL$ contains all Rickard idempotents except one. So, there exists a meet-prime smashing ideal $P$ such that $\Ker(\gG_P\ot -)=\loc\parens{\gG_Q}{Q\neq P}\subseteq \scL$, where the equality is by~\Cref{lem:form-ker}. Suppose that this containment relation is proper. Then there exists an object $X\in \scL$ such that $\gG_P\ot X\neq 0$. Since $\loc(\gG_P)$ is minimal, $\loc(\gG_P\ot X)=\loc(\gG_P)$. Moreover, $\gG_P\ot X\in \scL$, thus, $\gG_P\in \scL$. This forces $\scL=\scT$, leading to a contradiction. Uniqueness follows from the fact that $\gG_P\ot \gG_Q=0$ if and only if $P\neq Q$.
\end{proof}

\begin{cor}
\label{cor:obj-big-primes}%
Suppose that $\scT$ is stratified by the small smashing support. Then the big prime localizing ideals of $\scT$ coincide with the objectwise-prime localizing ideals of $\scT$ and there is only a set of such. In particular,
\[
\SPC=\set{\Ker(\gG_P\ot -)}{P\in \Spcs}.
\]
\end{cor}

\begin{proof}
The statement follows from~\Cref{thm:radicals}, \Cref{prop:objprimes} and~\cite[Lemma~4.1.7]{BalchinStevenson21}.
\end{proof}

\begin{cor}
\label{cor:spcs-objprimes}%
Suppose that $\scT$ is stratified by the small smashing support. Then the map $\Spcs \to \SPC,\, P\mapsto \Ker(\gG_P\ot -)$ is bijective.
\end{cor}

\begin{rem}
\label{rem:smash-obj}%
Let $P$ be a meet-prime smashing ideal of $\scT$. It is straightforward to verify that $\sigma(\Ker(\gG_P\ot -))=\Spcs\setminus\{P\}$. Utilizing this relation and~\Cref{rem:spcs-extend} leads to the following series of equivalences:
\[
\Ker(\gG_P\ot -)\in \sfS(\scT) \Leftrightarrow \Ker(\gG_P\ot -)=P \Leftrightarrow V_P=\{P\}.
\]
The first equivalence holds because if $\Ker(\gG_P\ot -)$ is smashing, then $\Ker(\gG_P\ot -)$ is contained in some $Q\in \Spcs$. Since $\gG_Q\notin Q$, it holds that $\gG_P\ot \gG_Q\neq 0$. Therefore, $P=Q$. Combined with the inclusion $P\subseteq \Ker(\gG_P\ot -)$, we obtain $\Ker(\gG_P\ot -)=P$. This also explains the implication $\Ker(\gG_P\ot -)=P\Rightarrow V_P=\{P\}$. If $V_P=\{P\}$, then $\gG_P=f_P$. Thus, $\Ker(\gG_P\ot -)=\Ker(f_P\ot -)=P$.
\end{rem}

\begin{cor}
Suppose that $\scT$ is stratified by the small smashing support. Then the smashing objectwise-prime localizing ideals of $\scT$ correspond to the closed points, with respect to the standard topology, of $\Spcs$.
\end{cor}

\begin{proof}
The claim follows from~\Cref{cor:spcs-objprimes} and~\Cref{rem:smash-obj}.
\end{proof}

We conclude with an observation about the small and big smashing supports.
\begin{prop}
\label{lem:supports-right-idempotents}%
The small and big smashing supports coincide, i.e., $\supps(X)=\Supps(X),\, \forall X\in \scT$, if and only if every point of $\Spcs$, with respect to the standard topology, is a closed point, i.e., $\Spcs$ is $T_1$.
\end{prop}

\begin{proof}
It holds: $\supps=\Supps \Leftrightarrow \forall X\in \scT,\, \forall P\in \Spcs\colon \gG_P\ot X\neq 0 \Leftrightarrow X\notin P$. Equivalently, $P=\Ker(\gG_P\ot -),\ \forall P\in \Spcs$. By~\Cref{rem:smash-obj}, this statement holds if and only if every point of $\Spcs$ is a closed point.
\end{proof}

\begin{ex}
\label{ex:spcs-closed-points}%
Let $R$ be a commutative absolutely flat ring. Then $\sfD(R)$ satisfies the Telescope Conjecture; see~\cite{Stevenson14,BazzoniStovicek17}. So, $\Spcsa{\sfD(R)}\cong \Spec(R)^\vee$ is $T_1$, since, in this particular case, $\Spec(R)$ is Hausdorff (and $\Spec(R)\cong \Spec(R)^\vee$). Hence, the big and small smashing supports on $\sfD(R)$ coincide.
\end{ex}

\section{Smashing localizations}
\label{sec:local}%
Recall our assumption that~\Cref{hyp:spatial} holds. The aim of this section is to prove that if a big tt-category $\scT$ satisfies the local-to-global principle and each meet-prime smashing localization $\scT/P$ is stratified by the small smashing support, then $\scT$ is stratified by the small smashing support. If $\scT$ satisfies the Telescope Conjecture, this recovers~\cite[Corollary 5.3]{BarthelHeardSanders21a}. The local-to-global principle hypothesis cannot be dropped. Further, we show that if $\Spcs$ admits a cover by finitely many closed subsets such that each corresponding smashing localization is stratified, then $\scT$ is stratified.

\subsection{Meet-prime smashing localizations}
\begin{lem}
\label{lem:image-of-ideal}%
Let $F_1\colon \scC\to \scD$ be a coproduct-preserving tensor-triangulated functor and $F_2\colon \scC\to \scC$ a coproduct-preserving triangulated endofunctor of $\scC$ such that $F_2(X\ot Y)\cong X\ot F_2Y$. Then for every collection of objects $\scX\subseteq\scC$, if $X\in \loc(\scX)$, then $F_iX\in \loc(F_i(\scX)),\ i=1,2$.
\end{lem}

\begin{proof}
The collection $\set{Y\in \scC}{F_iY\in \loc(F_i(\scX))}$ is a localizing ideal of $\scC$ that contains $\scX$. Hence, it contains $\loc(\scX)$ and, as a result, $F_iX\in \loc(F_i(\scX))$.
\end{proof}

\begin{rem}
\label{rem:gen-image-of-ideal}%
A generalized version of the second case of~\Cref{lem:image-of-ideal} is the following: Let $\scC$ be a tensor-triangulated category that acts on two triangulated categories $\scD_1$ and $\scD_2$, in the sense of~\cite{Stevenson13}, with the actions denoted by $\ast_1$ and $\ast_2$, respectively. If $F\colon \scD_1 \to \scD_2$ is a coproduct-preserving triangulated functor that preserves the action of $\scC$, i.e., $F(X\ast_1 Y)\cong X\ast_2 FY,\ \forall X\in \scC,\ \forall Y\in \scD_1$ and $\scX$ is a collection of objects of $\scD_1$, then $F(\mathsf{loc}^{\ast_1}(\scX))\subseteq \mathsf{loc}^{\ast_2}(F(\scX))$.
\end{rem}

\begin{rec}
\label{rem:ideals-of-quotient}%
Let $\scL$ be a localizing ideal of $\scT$. Then the localizing ideals of $\scT/\scL$ stand in bijection with the localizing ideals of $\scT$ that contain $\scL$. More precisely, if
$j_\scL\colon \scT \to \scT/\scL$ is the quotient functor, then the map that takes a localizing ideal $\scR\subseteq \scT/\scL$ to $j_\scL^{-1}(\scR)$ is a bijection with inverse given by taking direct images of localizing ideals of $\scT$ under $j_\scL$. If $j_\scL$ has a right adjoint $k_\scL$, then there is a projection formula: $k_\scL(j_\scL (X)\ot Y)\cong X\ot k_\scL(Y),\ \forall X\in \scT,\ Y\in \scT/\scL$~\cite{BalmerDellambrogioSanders16}.
\end{rec}

\begin{rem}
\label{rem:inverse-image-localization-functor}%
In keeping with the notation of~\Cref{rem:ideals-of-quotient}, if $\scL$ is a smashing ideal, then $k_\scL j_\scL=-\ot f_\scL$. In addition, if $\scR$ is a localizing ideal that contains $\scL$, then $(-\ot f_\scL)^{-1}(\scR)=\scR$. Indeed, if $X\ot f_\scL\in \scR$, then tensoring the idempotent triangle corresponding to $\scL$ with $X$ yields $e_\scL\ot X\to X\to f_\scL\ot X$. Since $\scL\subseteq \scR$, it follows that $e_\scL\ot X\in \scR$, thus $X\in \scR$. Since $\scR$ is an ideal, the converse inclusion also holds.
\end{rem}

Let $P$ be a smashing ideal of $\scT$ (not necessarily meet-prime). Then $\scT/P$ is a big tt-category and the quotient functor $j_P\colon \scT \to \scT/P$ is an essentially surjective coproduct-preserving tt-functor with a fully faithful right adjoint $k_P\colon \scT/P\to \scT$ that preserves coproducts, since $j_P$ preserves rigid=compact objects. Therefore, $j_P$ induces an injective continuous map $f\colon \Spcsl{P}\to \Spcs$. By identifying $\Spcsl{P}$ with $V_P$, the induced map $f$ is identified with the inclusion $V_P\hookrightarrow \Spcs$.~If $\Spcs$ is $T_D$, then $\Spcsl{P}$ is $T_D$, since being $T_D$ is a hereditary topological property.~If $Q\in \Spcsl{P}$, i.e., $Q\supseteq P$, then the corresponding Rickard idempotent is $j_P(\gG_Q)$.
\begin{lem}
\label{lem:vanishing-of-rickards}%
Let $P\in \sfS(\scT)$ and $Q \in \Spcs$ such that $P\nsubseteq Q$. Then $j_P(\gG_Q)=0$.
\end{lem}

\begin{proof}
Let $\scS$ be a smashing ideal such that $\{Q\}= U_\scS\cap V_Q$ (recall that $\Spcs$ is assumed $T_D$). Since $P\nsubseteq Q$, we have $U_{P\cap \scS}\cap V_Q=U_P\cap U_\scS\cap V_Q=\{Q\}$. This means that $\gG_Q=e_{P\cap \scS}\ot f_Q=e_P\ot e_\scS\ot f_Q$. So, $j_P(\gG_Q)=j_P(e_P)\ot j_P(e_\scS)\ot j_P(f_Q)=0$, due to the fact that $e_P\in P$.
\end{proof}

\begin{prop}
\label{prop:ltg-quotients}%
Suppose that $\scT$ satisfies the local-to-global principle. Then $\scT/P$ satisfies the local-to-global-principle, for every $P\in \sfS(\scT)$.
\end{prop}

\begin{proof}
Since $\scT$ satisfies the local-to-global principle, $1_\scT\in \loc\paren{\gG_Q}{Q\in \Spcs}$. So, $1_{\scT/P}=j_P(1_\scT)\in \loc\paren{j_P(\gG_Q)}{Q\in \Spcs}=\loc\parens{j_P(\gG_Q)}{Q\in V_P}$. Here we used~\Cref{lem:image-of-ideal} and~\Cref{lem:vanishing-of-rickards}. Thus, $\scT/P=\loc\parens{j_P(\gG_Q)}{Q\in V_P}$, meaning that $\scT/P$ satisfies the local-to-global principle. 
\end{proof}

\begin{lem}
\label{lem:inverse-image-of-ideal-generated-by-image}%
Let $P$ be a smashing ideal of $\scT$. Then
\[
j_P^{-1}(\loc(j_P(X)))=\loc(e_P,X),
\]
where $j_P\colon \scT\to \scT/P$ is the quotient functor.
\end{lem}

\begin{proof}
It is clear that $\loc(e_P,X)\subseteq j_P^{-1}(\loc(j_P(X)))$. Since $k_Pj_P(X)\cong f_P\ot X$, it follows that $j_P(X)\in k_P^{-1}(\loc(e_P,X))$. Thus, $\loc(j_P(X))\subseteq k_P^{-1}(\loc(e_P,X))$. Therefore, $j_P^{-1}(\loc(j_P(X)))\subseteq (k_Pj_P)^{-1}(\loc(e_P,X))=\loc(e_P,X)$, with the last equality by~\Cref{rem:inverse-image-localization-functor}.
\end{proof}

\begin{rem}
\label{rem:Rickard-left-idem}%
Let $j_P\colon \scT \to \scT/P$ be the quotient functor, where $P$ is a meet-prime smashing ideal.~Let $\scS\in \sfS(\scT)$ such that $U_\scS\cap V_P=\{P\}$.~Applying $j_P$ to the idempotent triangle corresponding to $P$ yields $j_P(e_P)\to 1_{\scT/P} \to j_P(f_P)$. It follows that $j_P(f_P)\cong 1_{\scT/P}$, so $j_P(\gG_P)\cong j_P(e_\scS)$. Conclusion: $j_P(\gG_P)$ is a left idempotent.
\end{rem}

\begin{prop}
\label{prop:minimality-reduction}%
Let $P$ be a meet-prime smashing ideal of $\scT$. Then $\loc(\gG_P)$ is minimal in $\Loc(\scT)$ if and only if $\loc(j_P(\gG_P))$ is minimal in $\Loc(\scT/P)$.
\end{prop}

\begin{proof}
$(\Rightarrow)$ Suppose that $\loc(\gG_P)$ is minimal and let $X\in \loc(j_P(\gG_P))$ be a non-zero object. Since $j_P(\gG_P)$ is a left idempotent, $j_P(\gG_P)\ot X\neq 0$. Write $X=j_P(Y)$, for some $Y\in \scT$. Then $j_P(\gG_P\ot Y)=j_P(\gG_P)\ot X\neq 0$, thus $\gG_P\ot Y\neq 0$. It follows, by minimality of $\loc(\gG_P)$, that $\loc(\gG_P\ot Y)=\loc(\gG_P)$. Hence, $\gG_P\in \loc(Y)$. Invoking~\Cref{lem:image-of-ideal} for the functor $j_P$ results in $j_P(\gG_P)\in \loc(j_P(Y))=\loc(X)$. Conclusion: $\loc(j_P(\gG_P))$ is minimal.

$(\Leftarrow)$ Suppose that $\loc(j_P(\gG_P))$ is minimal. Then $P\subsetneq j_P^{-1}(\loc(j_P(\gG_P)))$, which is minimal over $P$. By~\Cref{lem:inverse-image-of-ideal-generated-by-image}, this reads $P\subsetneq \loc(e_P,\gG_P)$. Pick a non-zero object $X\in \loc(\gG_P)$. Since $\loc(\gG_P)\subseteq \loc(f_P)=\im(-\ot f_P)$, it follows that $X\ot f_P\cong X\neq 0$, i.e., $X\notin P$. This shows that $\loc(\gG_P)\cap  P=0$. Therefore, $P\subsetneq \loc(e_P,X)\subseteq \loc(e_P,\gG_P)$. Since $\loc(e_P,\gG_P)$ is minimal over $P$, it follows that $\loc(e_P,X)= \loc(e_P,\gG_P)$, so $\gG_P\in \loc(e_P,X)$. By~\Cref{lem:image-of-ideal}, $\gG_P\cong \gG_P\ot \gG_P\in \loc(\gG_P\ot e_P,\gG_P\ot X)=\loc(\gG_P\ot X)$. As a result, $\loc(\gG_P)=\loc(\gG_P\ot X)\subseteq \loc(X)$, so $\gG_P\in \loc(X)$. Conclusion: $\loc(\gG_P)$ is minimal.
\end{proof}

\begin{prop}
\label{prop:minimality-equivalence-quotients}%
The category $\scT$ satisfies minimality if and only if $\scT/P$ satisfies minimality, for every $P\in \Spcs$.
\end{prop}

\begin{proof}
Suppose that $\scT$ satisfies minimality. Let $P\in \Spcs$ and $Q\in \Spcsl{P}$, i.e., $Q\in V_P$ and consider a non-zero object $j_P(X)\in \loc(j_P(\gG_Q))$. The fact that $\loc(j_P(\gG_Q))\cap j_P(Q)=0$ leads to $X\notin Q$. Now ponder the quotient $j_Q\colon \scT\to \scT/Q$. The ideal $\loc(\gG_Q)$ is minimal by assumption, so by~\Cref{prop:minimality-reduction}, $\loc(j_Q(\gG_Q))$ is minimal. Equivalently, $\loc(e_Q,\gG_Q)$ is minimal over $Q$. Since $X\notin Q$, it follows that $\loc(e_Q,X)=\loc(e_Q,\gG_Q)$. Invoking~\Cref{lem:supp-gen} yields $\supps(e_Q)\cup \supps(X)=\supps(e_Q)\cup \supps(\gG_Q)$. Thus, $U_Q\cup \supps(X)=U_Q\cup \{Q\}$. Since $Q\notin U_Q$, we infer that $Q\in \supps(X)$, which means that $\gG_Q\ot X\neq 0$. Further, minimality of $\loc(\gG_Q)$ implies that $\gG_Q\in \loc(X)$. In consequence, $j_P(\gG_Q)\in \loc(j_P(X))$, proving that $\loc(j_P(\gG_Q))$ is minimal. The converse implication is given by~\Cref{prop:minimality-reduction}.
\end{proof}

\begin{rem}
The right-hand implication in~\Cref{prop:minimality-equivalence-quotients} holds without assuming that $P$ is necessarily meet-prime.
\end{rem}
\begin{cor}
\label{prop:minimality-everywhere}%
Suppose that $\scT$ satisfies the local-to-global principle. Then $\scT$ is stratified by the small smashing support if and only if $\scT/P$ is stratified by the small smashing support, for every $P\in \Spcs$.
\end{cor}

\begin{proof}
Combine~\Cref{prop:ltg-quotients} and~\Cref{prop:minimality-equivalence-quotients}.
\end{proof}

\begin{rem}
\label{rem:strat-converse-counterexample}%
If $\scT$ is not assumed to a-priori satisfy the local-to-global principle, the converse of~\Cref{prop:minimality-everywhere} does not hold in general. For instance, if $R$ is an absolutely flat ring that is not semi-artinian, then $\sfD(R)$ does not satisfy the local-to-global principle, even though its localizations $\sfD(k(\mfp))$ (where $k(\mfp)$ is the residue field at $\mfp \in \Spec(R)$) are stratified; see~\cite[Theorem~4.7]{Stevenson14}.
\end{rem}

\subsection{Stratification and closed covers}
\begin{prop}\label{prop:ltg-finite-cover}%
Suppose that $\Spcs=\bigcup V_{\scS_i}$, where $\{\scS_i\}$ is a finite set of smashing ideals, and assume that each $\scT/\scS_i$ satisfies the local-to-global principle. Then $\scT$ satisfies the local-to-global principle.
\end{prop}

\begin{proof}
By an easy induction argument, it suffices to prove the statement in the case $\Spcs=V_{\scS_1}\cup V_{\scS_2}$. Consider the quotient functor $j_{\scS_i}\colon \scT \to \scT/\scS_i$, where $i=1,2$. Since $\scT/\scS_i$ satisfies the local-to-global principle and $\Spcsa{\scT/\scS_i}\cong V_{\scS_i}$, it holds that $j_{\scS_i}(1)=1\in \loc\parens{j_{\scS_i}(\gG_P)}{P\in V_{\scS_i}}$. By~\Cref{lem:inverse-image-of-ideal-generated-by-image}, it follows that $1\in \loc(e_{\scS_i},\{\gG_P\}_{P\in V_{\scS_i}})$. If $P\in V_{\scS_i}$, i.e., $P\supseteq \scS_i$, then $f_P\ot f_{\scS_i}\cong f_P$. Thus, $\gG_P\ot f_{\scS_i}\cong \gG_P$. Invoking~\Cref{lem:image-of-ideal}, we have $f_{\scS_i}\in \loc\parens{\gG_P}{ P\in V_{\scS_i}}$. Therefore, $f_{\scS_i}\in \loc\paren{\gG_P}{P\in \Spcs}$. So, $f_{\scS_1}\oplus f_{\scS_2},\ f_{\scS_1}\ot f_{\scS_2}\in \loc\paren{\gG_P}{P\in \Spcs}$. From the Mayer--Vietoris triangle $f_{\scS_1\cap \scS_2} \to f_{\scS_1}\oplus f_{\scS_2}\to f_{\scS_1}\ot f_{\scS_2}$, see~\cite[Theorem 3.13]{BalmerFavi11}, it follows that $f_{\scS_1\cap \scS_2}\in \loc\paren{\gG_P}{P\in \Spcs}$. Since $V_0=\Spcs=V_{\scS_1}\cup V_{\scS_2}=V_{\scS_1\cap \scS_2}$, we have $\scS_1\cap \scS_2=0$. Consequently, $f_{\scS_1\cap \scS_2}=1$. In conclusion, $1\in \loc\paren{\gG_P}{P\in \Spcs}$, which proves that $\scT$ satisfies the local-to-global principle.
\end{proof}

\begin{prop}\label{prop:closed-covers}%
Suppose that $\Spcs=\bigcup V_{\scS_i}$, where $\{\scS_i\}$ is a set of smashing ideals, and assume that each $\scT/\scS_i$ satisfies minimality. Then $\scT$ satisfies minimality.
\end{prop}

\begin{proof}
By assumption, if $P\in \Spcs$, then $P$ lies in some $V_{\scS_i}$. In other words, there exists some $\scS_i$ such that $\scS_i\subseteq P$. The category $\scT/P$ can be realized as a localization of $\scT/\scS_i$, as in the following commutative diagram:
\[
\begin{tikzcd}[row sep=4em]
\scT \rar[two heads] \dar[two heads] & \scT/\scS_i \dar[two heads] \dlar[two heads]
\\
\scT/P \rar[phantom,"\simeq"] & (\scT/\scS_i)/j_{\scS_i}(P).
\end{tikzcd}
\]
Since $\scT/\scS_i$ satisfies minimality, it follows by~\Cref{prop:minimality-equivalence-quotients} that $\scT/P$ satisfies minimality. In conclusion, $\scT/P$ satisfies minimality, for every $P\in \Spcs$. So, again by~\Cref{prop:minimality-equivalence-quotients}, $\scT$ satisfies minimality.
\end{proof}

\begin{cor}\label{cor:strat-closed-covers}%
Suppose that $\Spcs=\bigcup V_{\scS_i}$, where $\{\scS_i\}$ is a finite set of smashing ideals, and assume that each $\scT/\scS_i$ is stratified by the small smashing support. Then $\scT$ is stratified by the small smashing support.
\end{cor}

\begin{rem}
The case of the trivial cover $\Spcs=\bigcup_{P\in \Spcs}V_P$ in~\Cref{prop:closed-covers} recovers the statement of~\Cref{prop:minimality-equivalence-quotients}. In~\Cref{prop:ltg-finite-cover}, if $\{\scS_i\}$ is allowed to be an infinite set, then the most we can deduce is that $\scT/P$ satisfies the local-to-global principle, for every $P\in \Spcs$. As we have already seen in~\Cref{rem:strat-converse-counterexample}, this is not enough to guarantee that $\scT$ satisfies the local-to-global principle.
\end{rem}

\begin{rem}
Results of similar flavor appear in~\cite{BarthelHeardSanders21a}, where $\Spc$ is covered by complements of Thomason subsets and the notion of support considered is the one introduced in~\cite{BalmerFavi11}. Note that the cover $\{V_{\scS_i}\}$ of $\Spcs$ in the above results consists of closed subsets, which are complements of open subsets. This should not come as a surprise since, by Stone duality, smashing ideals of $\scT$ correspond to open subsets of $\Spcs$, while thick ideals of $\scT^\rmc$ correspond to Thomason subsets of $\Spc$.
\end{rem}

\section{Comparison maps}
\label{sec:comp}%

Let $\scT$ be a big tt-category such that~\Cref{hyp:spatial} holds and whose smashing spectrum is $T_D$ and assume that $\scT$ is stratified by the small smashing support. Let~$\scB$ be a homological prime. By~\cite[Lemma~5.2.1]{BalchinStevenson21}, $\chi(\scB)\coloneqq h^{-1}(\scB')=\Ker[-,I_\scB]$ is a big prime localizing ideal of $\scT$, where $\scB'$ is the unique maximal localizing ideal of $\Modc$ that contains $\scB$ and $I_\scB$ is the associated pure-injective weak ring; see~\Cref{cons:weak-rings}.~\Cref{cor:obj-big-primes} asserts that $\chi(\scB)=\Ker(\gG_P\ot -)$, for a unique $P\in \Spcs$. This produces a well-defined map $\xi\colon \Spch \to \Spcs$ that associates each $\scB\in \Spch$ with the unique $P\in \Spcs$ such that $\chi(\scB)=\Ker(\gG_P\ot -)$.

By construction, $\xi$ is the composite $\Spch\xr{\chi}\SPC\xr{\cong} \Spcs$, where the second map is the inverse of the map that takes $P\in \Spcs$ to $\Ker(\gG_P\ot -)$; see~\Cref{cor:spcs-objprimes}.
\begin{lem}
\label{lem:xi-inj}%
The map $\xi\colon \Spch\to \Spcs$ is injective.
\end{lem}

\begin{proof}
It suffices to show that $\chi \colon \Spch \to \SPC$ is injective. To this end, let $\scB$ and $\scC$ be two distinct homological primes. According to~\cite[Corollary~4.9]{Balmer20b}, $[I_\scB,I_\scC]=0$. Thus, $I_\scB\in \chi(\scC)$. Since $I_\scB\notin \chi(\scB)$, it follows that $\chi(\scB)\neq \chi(\scC)$.
\end{proof}

\begin{lem}
\label{lem:xi-cont}%
Let $X$ be an object of $\scT$. Then $\xi^{-1}(\supps(X))=\supph(X)$.
\end{lem}

\begin{proof}
Let $\scB$ be a homological prime and consider the associated pure-injective weak ring $I_\scB\in \scT$. Then $\chi(\scB)=\Ker[-,I_\scB]=\Ker(\gG_{\xi(\scB)}\ot -)$. It follows from the definition of $\supps$ and $\supph$ that $\xi^{-1}(\supps(X))=\supph(X)$.
\end{proof}

\begin{rem}
\label{rem:compare-bf-strat}%
If $\Spc$ is weakly noetherian and $\scT$ is stratified by the Balmer--Favi support, then the comparison map $\phi\colon \Spch\to \Spc$ is a homeomorphism; see~\cite[Theorem~4.7]{BarthelHeardSanders21b}. In our case, we had to assume that $\scT$ is stratified by the small smashing support (with the analogous topological hypothesis being that $\Spcs$ is $T_D$) to even obtain the map $\xi\colon \Spch\to \Spcs$. This, however, is not enough to guarantee that $\xi$ is bijective. The failure of surjectivity of $\xi$ is measured by the ``kernel" of the homological support.
\end{rem}

\begin{prop}[cf. {\cite[Proposition~3.14]{BarthelHeardSanders21b}}]
\label{prop:xi-surj-supph}%
Let $\scT$ be a big tt-category whose smashing spectrum is $T_D$ and assume that $\scT$ is stratified by the small smashing support. The following are equivalent:
\begin{enumerate}[\rm(a)]
\item
\label{cond:xi-surj}%
$\xi$ is surjective.
\medbreak
\item
\label{cond:xi-direct-image}%
$\xi(\supph(X))=\supps(X),\ \forall X\in \scT$.
\medbreak
\item
\label{cond:supph-vanishing}%
$\supph$ detects vanishing of objects.
\end{enumerate}
\end{prop}

\begin{proof}
$(\rma)\Rightarrow (\rmb)$
Suppose that $\xi$ is surjective. \Cref{lem:xi-inj} implies that $\xi$ is bijective. The statement now follows by applying $\xi$ to the relation obtained in~\Cref{lem:xi-cont}.

$(\rmb)\Rightarrow (\rmc)$
Suppose that $\supph(X)=\varnothing$. It follows, by the assumption we made in~\eqref{cond:xi-direct-image}, that $\supps(X)=\varnothing$. Since $\scT$ is stratified, by~\Cref{prop:strat-class}, $\scT$ satisfies the local-to-global principle; so, $\supps$ detects vanishing. As a result, $X=0$.

$(\rmc)\Rightarrow (\rma)$
As a special case of~\Cref{lem:xi-cont}, for $X=\gG_P$: $\xi^{-1}(\{P\})=\supph(\gG_P)$. Hence, $\xi$ is surjective if and only if $\supph(\gG_P)\neq \varnothing,\ \forall P\in \Spcs$. 
Since $\gG_P\neq$~$0$ and $\supph$ is assumed to detect vanishing, $\supph(\gG_P)\neq \varnothing$; proving that $\xi$ is surjective.
\end{proof}

\subsection{The four spectra}\label{subsec:4spec}%
We continue to assume that $\scT$ is a big tt-category whose smashing spectrum is $T_D$ and that $\scT$ is stratified by the small smashing support. The spaces $\Spc$, $\Spch$, $\Spcs$, $\SPC$ are related via the following commutative diagram.

\begin{equation*}
\begin{tikzcd}[cells={text width={width("$\Spch$")},align=center},row sep=3.5em,column sep=0.8em]
& \Spch \dlar["\xi"',tail] \drar["\phi",two heads]
\arrow[tail,dd,outer sep=-11pt,"\chi"',bend right=90,shift right=19,shorten >=-4.7em,shorten <=-4.7em]
\\
\Spcs \arrow[rr,"\psi",two heads] \arrow[dr,shift left=1ex,"\rotatebox{315}{$\cong$}"',outer sep=-3pt] && \Spc
\\
& \SPC \arrow[ul,shift left=1ex] \arrow[ur,"\omega"',two heads]
\end{tikzcd}
\end{equation*}

To begin with, let $\scB$ be a homological prime such that $\chi(\scB)=\Ker(\gG_P\ot -)$. Using the relation $\Ker(\gG_P\ot -)\cap \scT^\rmc = P\cap \scT^\rmc$ yields $\phi(\scB)=h^{-1}(\scB)\cap \scT^\rmc=h^{-1}(\scB')\cap \scT^\rmc=\chi(\scB)\cap \scT^\rmc=\xi(\scB)\cap \scT^\rmc=(\psi \circ \xi)(\scB)$. This shows that $\psi \circ \xi = \phi$ and also that $\psi$ is equal to the composite $\Spcs \xr{\cong} \SPC \xr{\omega} \Spc$, where $\omega$ maps an objectwise-prime localizing ideal to its compact part. Lastly, the triangle on the left commutes by construction of $\xi$.

\begin{thm}
\label{thm:tc}%
Let $\scT$ be a big tt-category whose smashing spectrum is $T_D$ and assume that $\scT$ is stratified by the small smashing support. Then $\scT$ satisfies the Telescope Conjecture if and only if $\,\Spch$ is $T_0$ and $\supph$ detects vanishing of objects.
\end{thm}

\begin{proof}
By~\cite[Corollary~5.1.6]{BalchinStevenson21}, $\scT$ satisfies the Telescope Conjecture if and only if $\psi$ is bijective. Since $\xi$ is injective and $\phi$ is surjective and $\psi\circ \xi = \phi$, it holds that $\psi$ is bijective if and only if $\xi$ and $\phi$ are bijective. According to~\Cref{lem:xi-inj} and~\Cref{prop:xi-surj-supph}, $\xi$ is bijective if and only if $\supph$ detects vanishing of objects. By~\cite[Proposition~4.5]{BarthelHeardSanders21b}, $\phi$ is bijective if and only if $\Spch$ is $T_0$. Hence, $\scT$ satisfies the Telescope Conjecture if and only if $\Spch$ is $T_0$ and $\supph$ detects vanishing of objects.
\end{proof}

\section{Smashing stratification vs Balmer--Favi stratification}\label{sec:comparison-BF}%
We recall a few definitions, following the terminology used in~\cite{BarthelHeardSanders21a}. A point $\mfp\in \Spc$ is called \emph{weakly visible} if $\{\mfp\}$ can be written as an intersection of a Thomason subset and the complement of a Thomason subset. The Balmer spectrum $\Spc$ is called \emph{weakly noetherian} if all of its points are weakly visible. The \emph{generalization closure} of a point $\mfp\in \Spc$ is $\mathsf{gen}(\mfp)=\set{\mfq\in \Spc}{\mfp \in \ol{\{\mfq\}}}=\set{\mfq\in \Spc}{\mfp \subseteq \mfq}$. The Balmer spectrum is called \emph{generically noetherian} if $\mathsf{gen}(\mfp)$ is a noetherian space, for all $\mfp\in \Spc$. If $\Spc$ is generically noetherian, then $\Spc$ is weakly noetherian.

Assume that $\Spc$ is weakly noetherian. Every Thomason subset $V\subseteq \Spc$ corresponds to a thick ideal of $\scT^\rmc$, namely $\scT^\rmc_V=\set{x\in \scT^\rmc}{\supp(x) \subseteq V}$. Being compactly generated, the localizing ideal $\loc(\scT^\rmc_V)$ is a smashing ideal. The corresponding left and right idempotents are denoted by $e_V$ and $f_V$, respectively. Since we assumed that $\Spc$ is weakly noetherian, for each $\mfp\in \Spc$, there exist Thomason subsets $V_1,V_2\subseteq \Spc$ such that $\{\mfp\}=V_1\cap (\Spc \setminus V_2)$. We define $g_\mfp=e_{V_1}\ot f_{V_2}$. The objects $g_\mfp$ are idempotent and do not depend on the choice of Thomason subsets used to construct them. For any object $X\in \scT$, the \emph{Balmer--Favi support} of $X$ is $\Supp(X)=\set{\mfp\in \Spc}{g_\mfp\ot X\neq 0}$. For more details, see~\cite{BalmerFavi11}.

\begin{rem}\label{rem:td-wn}%
If $X$ is a spectral space, then $X$ is weakly noetherian if and only if the Hochster dual space $X^\vee$ (whose open subsets are the Thomason subsets of $X$) is $T_D$. In particular, this applies to $\Spc$. Note that $\Spcs$ is likely not spectral; see~\cite[Remark 3.2.12]{BalchinStevenson21}.
\end{rem}

\begin{lem}
\label{lem:TC-Rickard-BF-idempotents}%
Suppose that $\Spcs$ is $T_D$ and that $\scT$ satisfies the Telescope Conjecture. Then $\gG_P=g_{P^\rmc},\ \forall P\in \Spcs$.
\end{lem}

\begin{proof}
Let $P\in \Spcs$ and $\scS\in \sfS(\scT)$ such that $\{P\}=U_\scS\cap V_P$. Since the Telescope Conjecture holds and $\Spcs$ is $T_D$, $\psi\colon \Spcs \to \mathsf{Spc}(\scT^\rmc)^\vee$ is a homeomorphism and $\Spc$ is weakly noetherian. It holds that $\{P^\rmc\}=\psi(U_\scS)\cap \psi(V_P)=\psi(U_\scS)\cap (\Spc\setminus \psi(U_P))$, with $\psi(U_\scS)$ and $\psi(U_P)$ being Thomason subsets of $\Spc$. The thick ideal corresponding to $\psi(U_\scS)$ is $\scT^\rmc_{\psi(U_\scS)}=\set{x\in \scT^\rmc}{\supp(x)\subseteq \psi(U_\scS)}$. For all $x\in \scT^\rmc$, $\psi^{-1}(\supp(x))=U_{\loc(x)}$. So, $\supp(x)\subseteq \psi(U_\scS)$ if and only if $U_{\loc(x)}\subseteq U_\scS$, with the latter being equivalent to $x\in \scS^\rmc$. Therefore, $\scT^\rmc_{\psi(U_\scS)}=\scS^\rmc$. Similarly, $\scT^\rmc_{\psi(U_P)}=P^\rmc$. Consequently, $\loc(\scT^\rmc_{\psi(U_\scS)})=\scS$ and $\loc(\scT^\rmc_{\psi(U_P)})=P$. We infer that $e_{\psi(U_\scS)}=e_\scS$ and $f_{\psi(U_P)}=f_P$ and as a result, $\gG_P=g_{P^\rmc}$.
\end{proof}

\begin{cor}
\label{cor:TC-supps-Supp-compatibility}%
Suppose that $\Spcs$ is $T_D$ and that $\scT$ satisfies the Telescope Conjecture. Then $\psi^{-1}(\Supp(X))=\supps(X)$, for all $X \in \scT$.
\end{cor}

\begin{thm}
\label{thm:supps-BF-strat}%
Let $\scT$ be a big tt-category.
\begin{enumerate}[\rm(a)]
\item If $\Spc$ is generically noetherian and the Balmer--Favi support stratifies $\scT$, then $\scT$ satisfies the Telescope Conjecture, $\Spcs$ is $T_D$ and the small smashing support stratifies $\scT$.
\label{statement:BF-implies-supps}%
\item If $\Spcs$ is $T_D$ and the small smashing support stratifies $\scT$ and $\scT$ satisfies the Telescope Conjecture, then the Balmer--Favi support stratifies $\scT$.
\label{statement:supps-implies-BF}%
\end{enumerate}
\end{thm}

\begin{proof}
If $\Spc$ is generically noetherian and $\scT$ is stratified by the Balmer--Favi support, then, by~\cite[Theorem 9.11]{BarthelHeardSanders21a}, $\scT$ satisfies the Telescope Conjecture. Thus, $\psi\colon \Spcs \to \mathsf{Spc}(\scT^\rmc)^\vee$ is a homeomorphism. Since $\Spc$ is weakly noetherian, $\mathsf{Spc}(\scT^\rmc)^\vee$ is $T_D$. Therefore, $\Spcs$ is $T_D$, so the small smashing support is defined. Invoking~\Cref{lem:TC-Rickard-BF-idempotents} completes the proof of~\eqref{statement:BF-implies-supps}.

If $\scT$ is stratified by the small smashing support and $\scT$ satisfies the Telescope Conjecture, then $\Spcs$ is a spectral space and, by~\Cref{cor:TC-supps-Supp-compatibility}, the hypotheses of~\cite[Theorem 7.6] {BarthelHeardSanders21a} are satisfied. We deduce that the Balmer--Favi support stratifies $\scT$. This proves~\eqref{statement:supps-implies-BF}.
\end{proof}

\begin{rem}
\label{rem:tc-recovers}%
Let us emphasize once more that if the Telescope Conjecture holds, the theory of stratification developed here recovers the theory of~\cite{BarthelHeardSanders21a}. More explicitly, if $\scT$ satisfies the Telescope Conjecture, then the following hold:
\begin{enumerate}[\rm(a)]
\item
The comparison map $\psi\colon \Spcs \to \mathsf{Spc}(\scT^\rmc)^\vee$ is a homeomorphism.
\item
$\Spcs$ is $T_D$ if and only if $\Spc$ is weakly noetherian.
\end{enumerate}
Assuming that $\Spcs$ is $T_D$:
\begin{enumerate}[\rm(a)]
\item[$(\mathrm{c})$] The small smashing support coincides with the Balmer--Favi support (under the identification of the two spectra via $\psi$).
\item[$(\mathrm{d})$]
The formulation of the local-to-global principle and of minimality in~\Cref{defn:ltg-min} coincides with those given in~\cite{BarthelHeardSanders21a}.
\item[$(\mathrm{e})$]
$\scT$ is stratified by the small smashing support if and only if $\scT$ is stratified by the Balmer--Favi support.
\end{enumerate}
\end{rem}

\begin{rem}
\label{rem:potential-of-sss}%
According to~\cite[Theorem 7.6]{BarthelHeardSanders21a}, any stratifying support theory $\sigma\colon \Ob(\scT)\to \scP(X)$ (where $X$ is a weakly noetherian spectral space) that satisfies three equivalent conditions must be ``isomorphic'' to the Balmer--Favi support theory and the latter has to stratify $\scT$ as well. If $\Spcs$ is $T_D$ and $\sigma=\supps$ stratifies~$\scT$, then the three equivalent conditions we alluded to are equivalent to the Telescope Conjecture. In~\cite[Theorem 9.11]{BarthelHeardSanders21a}, under the hypothesis that $\Spc$ is generically noetherian, it is proved that if $\scT$ is stratified by the Balmer--Favi support, then the Telescope Conjecture holds. In the case of the small smashing support, that proof can neither be reproduced nor is it expected that a different proof exists. Conclusion: The theory of smashing stratification has the potential to encompass a wider range of categories than the theory of Balmer--Favi stratification, since the Telescope Conjecture is necessary for a category to be stratified by the Balmer--Favi support, but it is probably not necessary for smashing stratification. To be clear, at the time of this writing, an example of a category that is stratified by the small smashing support and fails the Telescope Conjecture is not known.
\end{rem}

\section{Induced maps and descent}
\label{sec:image-of-SpcsF}%

In the first part of this section, we probe the image of the map between smashing spectra induced by a tensor-triangulated functor; see also~\cite{Balmer20b} for analogous results concerning homological spectra. In the second part, we present conditions under which stratification descends along tensor-triangulated functors. All big tt-categories involved are assumed to satisfy~\Cref{hyp:spatial}.

Let $F\colon \scT \to \scU$ be a coproduct-preserving tt-functor between big tt-categories. Then $F$ induces a map of frames $\sfS(\scT) \to \sfS(\scU),\ \scS \mapsto \loc(F\scS)$, which, via Stone duality, gives rise to a continuous map $\Spcsa{F}\colon \Spcsa{\scU}\to \Spcs$. Explicitly, $\Spcsa{F}$ acts by sending $Q\in \Spcsa{\scU}$ to $\bigvee\set{\scS\in \sfS(\scT)}{\scS\subseteq F^{-1}(Q)}$. Additionally, since $F$ preserves rigid=compact objects, there is an induced continuous map $\mathsf{Spc}(F)\colon \mathsf{Spc}(\scU^\rmc)\to \Spc$ that takes $\mfq \in \mathsf{Spc}(\scU^\rmc)$ to $F^{-1}(\mfq)\cap \scT^\rmc$.
\begin{rem}\label{rem:form-of-spcsF}%
The map $\Spcsa{F}$ does not behave in a way similar to the more classical $\mathsf{Spc}(F)$, namely by taking inverse images. For one, the formula for $\Spcsa{F}$ is given by Stone duality, as explained above. More concretely, there are cases where $F^{-1}(Q)$, for $Q\in \Spcsa{\scU}$, is not a smashing ideal. An example is the derived base change functor $\pi\colon \sfD(\mathbb{Z})\to \sfD(\mathbb{F}_p)$ for a prime number $p$, as demonstrated in~\cite[Example 3.4.5]{BalchinStevenson21}.
\end{rem}

\begin{lem}
\label{lem:Spcs-Spc-compatibility}%
The following square is commutative:
\[
\begin{tikzcd}[cells={text width={width("$\mathsf{Spc}(\scU^\rmc)$")},align=center},row sep=3.5em]
\Spcsa{\scU} \rar["\Spcsa{F}"] \dar["\psi_\scU"'] & \Spcs \dar["\psi_\scT"]
\\
\mathsf{Spc}(\scU^\rmc) \rar["\mathsf{Spc}(F)"] & \Spc \!.
\end{tikzcd}
\]
\end{lem}

\begin{proof}
Let $Q\in \Spcsa{\scU}$. Then
\begin{align*}
&\mfp_1\coloneqq(\psi_\scT \circ \Spcsa{F})(Q)=\bigvee\set{\scS\in \sfS(\scT)}{\scS\subseteq F^{-1}(Q)}\cap \scT^\rmc, \\
&\mfp_2\coloneqq(\mathsf{Spc}(F) \circ \psi_\scU)(Q)=F^{-1}(Q\cap \scU^\rmc)\cap \scT^\rmc=F^{-1}(Q)\cap \scT^\rmc.
\end{align*}
Clearly, $\mfp_1\subseteq \mfp_2$. For any $x\in \mfp_2$, it holds that $\loc(x)$ is a smashing ideal of $\scT$ and $\loc(x)\subseteq F^{-1}(Q)$. This shows that $\mfp_2\subseteq \mfp_1$, thus $\mfp_1=\mfp_2$.
\end{proof}

\begin{cor}
\label{cor:SpcsF-homeo-tc}%
Let $F\colon \scT\to \scU$ be a coproduct-preserving tt-functor between big tt-categories such that the induced map $\Spcsa{F}\colon \Spcsa{\scU}\to \Spcs$ is a homeomorphism. $\!\!$If $\scT$ satisfies the telescope conjecture, then $\scU$ satisfies the telescope conjecture and the induced map $\mathsf{Spc}(F)\colon \mathsf{Spc}(\scU^\rmc) \to \Spc$ is a homeomorphism.
\end{cor}

\begin{proof}
By~\cite[Corollary 5.1.6]{BalchinStevenson21}, the map $\psi_\scT$ is a homeomorphism. Therefore, by~\Cref{lem:Spcs-Spc-compatibility}, $\mathsf{Spc}(F)\circ \psi_\scU$ is a homeomorphism. This implies that $\psi_\scU$ is injective. By~\cite[Proposition 5.1.5]{BalchinStevenson21}, $\psi_\scU$ is surjective. Hence, $\psi_\scU$ is a homeomorphism and, again by~\cite[Corollary 5.1.6]{BalchinStevenson21}, $\scU$ satisfies the telescope conjecture. Since $\psi_\scU$ and $\mathsf{Spc}(F)\circ \psi_\scU$ are homeomorphisms, $\mathsf{Spc}(F)$ is a homeomorphism.
\end{proof}

\subsection{The image of $\Spcsa{F}$}
Let $F\colon \scT \to \scU$ be a coproduct-preserving tt-functor between big tt-categories. By Brown representability, $F\colon \scT\to \scU$ has a right adjoint $G\colon \scU \to \scT$. Since $F$ is monoidal, hence preserves compact objects, $G$ is lax-monoidal and preserves coproducts. Further, $F$ and $G$ are related by the \emph{projection formula}: $G(FX\ot Y)=X\ot GY$; see~\cite[Proposition 2.15]{BalmerDellambrogioSanders16}. For $Y=1$, we see that $GF(-)=G(1)\ot -$. Moreover, $G(1)\neq 0$, since $\Hom_\scU(1=F1,1)\cong \Hom_\scT(1,G(1))$.

\begin{defn}
\label{defn:maxspcs}%
Let $\scT$ be a big tt-category. The \emph{maximal smashing spectrum} of $\scT$ is the subspace $\mathsf{MaxSpc}^\rms(\scT) \subseteq \Spcs$ that consists of those meet-prime smashing ideals of $\scT$ that are maximal with respect to the inclusion relation of meet-prime smashing ideals.
\end{defn} 

\begin{lem}
\label{lem:maxspcs-properties}%
Let $\scT$ be a big tt-category. Then the following hold:
\begin{enumerate}[\rm(a)]
\item
\label{prop:maxspcs-non-empty-closed-points}%
The maximal smashing spectrum of $\scT$ is non-empty and consists of the closed points of $\Spcs$.
\item
\label{prop:max-in-Spcs-S(T)}%
A meet-prime smashing ideal $P$ is maximal in $\Spcs$ if and only if $P$ is maximal in $\sfS(\scT)$.
\item
\label{prop:Supps-supps-closed-points}%
Assuming that $\Spcs$ is $T_D$, the following holds:
\[
\mathsf{MaxSpc}^\rms(\scT)\cap \Supps(X)=\mathsf{MaxSpc}^\rms(\scT)\cap \supps(X),\, \forall X\in \scT.
\]
\end{enumerate}
\end{lem}

\begin{proof}
Properties~\eqref{prop:maxspcs-non-empty-closed-points} and~\eqref{prop:max-in-Spcs-S(T)} follow from the fact that every proper smashing ideal is contained in some meet-prime smashing ideal, which, moreover, is maximal. For the non-emptyness property, simply note that $0$ is smashing. Let $X$ be an object of $\scT$. If $P\in \Spcs$ is a closed point, then $P=\Ker(\gG_P\ot -)$; see~\Cref{rem:smash-obj}. Consequently, $P\in \Supps(X)$ if and only if $P\in \supps(X)$. This proves~\eqref{prop:Supps-supps-closed-points}.
\end{proof}
\begin{prop}
\label{prop:maxspcs-supp(G(1))}%
Let $F\colon \scT \to \scU$ be a coproduct-preserving tt-functor between big tt-categories with right adjoint $G$. Then $\mathsf{MaxSpc}^\rms(\scT) \cap \Supps(G(1)) \subseteq \im \Spcsa{F}$.
\end{prop}

\begin{proof}
Let $P\in \mathsf{MaxSpc}^\rms(\scT)$ such that $G(1)\notin P$ and consider the corresponding right idempotent $f_P$. Then $F(f_P)$ is a right idempotent whose corresponding smashing ideal is $\loc(F(P))$. Claim: $\loc(F(P))$ is proper. If this was not the case, then $F(f_P)=0$. Thus, $G(1)\ot f_P=GF(f_P)=0$. It follows that $G(1)\in P$, which is ruled out by assumption. Being a proper smashing ideal, $\loc(F(P))$ is contained in some $Q\in \Spcsa{\scU}$. Since $P\subseteq F^{-1}(\loc(F(P)))\subseteq F^{-1}(Q)$, it holds that $P\subseteq \Spcsa{F}(Q)$. Since $P$ is maximal, $\Spcsa{F}(Q)=P$. In conclusion, $P\in \im \Spcsa{F}$.
\end{proof}

\begin{prop}
\label{prop:img-subseq-supps}%
Let $F\colon \scT \to \scU$ be a coproduct-preserving tt-functor between big tt-categories with right adjoint $G$. Then $\im \Spcsa{F} \subseteq \Supps(G(1))$. Assuming that $\Spcs$ is $T_D$, if $G$ is conservative, i.e., $\Ker G=0$, then $\im \Spcsa{F}\subseteq \supps(G(1))$.
\end{prop}

\begin{proof}
Let $Q\in \Spcsa{\scU}$ and $P=\Spcsa{F}(Q)$. Since $P$ and $Q$ are smashing ideals, the localizations $\scT/P$ and $\scU/Q$ are big tt-categories. Let $j_P$ and $j_Q$ denote the corresponding quotient functors with right adjoints $k_P$ and $k_Q$, respectively. It holds that $P\subseteq F^{-1}(Q)=\Ker(j_Q\circ F)$. Therefore, there exists a unique triangulated functor $\wt{F}\colon \scT/P\to \scU/Q$ such that $\wt{F}\circ j_P \cong j_Q\circ F$. Moreover, $\wt{F}$ is monoidal and preserves coproducts. Hence, $\wt{F}$ has a right adjoint $\wt{G}$. Since $\wt{F}\circ j_P \cong j_Q\circ F \dashv G\circ k_Q$ and $\wt{F}\circ j_P \dashv k_P \circ \wt{G}$, we infer that $G\circ k_Q\cong k_P\circ \wt{G}$. Our discussion so far is recorded in the following diagram:
\[
\begin{tikzcd}[cells={text width={width("$\scU/Q.$")},align=center},row sep=4em,column sep=4em]
\scT \rar["j_P",two heads,shift left=1ex] \rar[phantom,"\rotatebox{90}{$\vdash$}"] \dar["F"',shift right=1ex] \dar[phantom,"\dashv"] & \scT/P \lar["k_P",shift left=1ex,hook'] \dar["\wt{F}"',shift right=1ex] \dar[phantom,"\dashv"]
\\
\scU \rar["j_Q",two heads,shift left=1ex] \rar[phantom,"\rotatebox{90}{$\vdash$}"] \uar["G"',shift right=1ex] & \scU/Q. \lar["k_Q",shift left=1ex,hook'] \uar["\wt{G}"',shift right=1ex]
\end{tikzcd}
\]
Applying $j_P$ to both sides of $G\circ k_Q\cong k_P\circ \wt{G}$, we obtain the relation $\wt{G}\cong j_P\circ G\circ k_Q$. As a result, $j_P(G(f_Q))\cong j_P(G(k_Q(1))) \cong \wt{G}(1)\neq 0$. This reads $G(f_Q)\notin P$. In particular, $G(f_Q)\neq 0$. The next piece of information we need is that the morphism adjoint to the right idempotent $1 \to f_Q$, i.e., $1 \xr{\eta} G(1) \to G(f_Q)$, where $\eta$ is the unit of adjunction, is a weak ring. Tensoring this composite with $G(f_Q)$ results in a split monic $G(f_Q) \to G(1)\ot G(f_Q) \to G(f_Q)\ot G(f_Q)$. It follows that $G(f_Q) \to G(1)\ot G(f_Q)$ is split monic. So, $G(1)\ot G(f_Q)\neq 0$ since it admits the non-zero object $G(f_Q)$ as a summand. Finally, suppose that $G(1)\in P$. Then $G(1)\in F^{-1}(Q)$. This implies that $FG(1)\ot f_Q=0$. By the projection formula, $G(1)\ot G(f_Q)=0$, which leads to a contradiction. We conclude that $G(1)\notin P$, i.e., $P\in \Supps(G(1))$.

Now assume that $\Ker G=0$. Claim: $\Ker \wt{G}=0$.  If $X\in \scU$ and $\wt{G}(j_Q(X))=0$, then $G(f_Q\ot X)=G(k_Q(j_Q(X)))=k_P(\wt{G}(j_Q(X)))=0$. Therefore, $f_Q\ot X=0$, which means that $X\in Q$, so $j_Q(X)=0$. This proves the claim, which implies that $\Ker (\wt{G}(1)\ot -)=\Ker \wt{G}\wt{F}=\Ker \wt{F}$. Now let $\{P\}=U_\scS\cap V_P$, so that $\gG_P=e_\scS\ot f_P$. Since $e_\scS\notin P$ and $P=\Spcsa{F}(Q)$, it follows that $F(e_\scS)\notin Q$. Thus, $\wt{F}(j_P(e_\scS))=j_Q(F(e_\scS))\neq 0$. In other words, $j_P(e_\scS)\notin \Ker \wt{F}$. As a result, $\gG_P\ot G(f_Q)=\gG_P\ot G(k_Q(1))=k_P(j_P(e_\scS)\ot \wt{G}(1))\neq 0$, with the second equality by using the relation $j_P(\gG_P)=j_P(e_\scS)$ and the projection formula for $j_P\dashv k_P$. Tensoring the split monic $G(f_Q) \to G(1)\ot G(f_Q)$ with $\gG_P$, we conclude that $\gG_P\ot G(1)\neq 0$, so $P\in \supps(G(1))$.
\end{proof}

\begin{cor}
\label{cor:maxspcs-int-image-equals-maxspcs-int-supps}%
Let $F\colon \scT \to \scU$ be a coproduct-preserving tt-functor between big tt-categories with right adjoint $G$. Then $\mathsf{MaxSpc}^\rms(\scT)\cap \im \Spcsa{F}=\mathsf{MaxSpc}^\rms(\scT) \cap \Supps(G(1))$.
\end{cor}

\begin{rem}
\label{rem:img-not-equal-Supps}%
The inclusion $\im \Spcsa{F}\subseteq \Supps(G(1))$ in~\Cref{prop:img-subseq-supps} is not an equality in general. For instance, let $P\in \Spcs$ and $\scS\in \sfS(\scT)\setminus \{\scT\}$ such that $P\subsetneq \scS$. Then $\supps(f_\scS)=V_\scS\neq \Supps(f_\scS)$, since the former does not contain $P$ ($\gG_P\ot f_\scS = f_\scS \ot e_\scS \ot f_P=0$) while the latter does ($f_\scS\notin P$). Let $j_\scS\colon \scT \to \scT/\scS$ be the quotient functor and $k_\scS$ its right adjoint. Then $\im \Spcsa{j_\scS}=V_\scS=\supps(f_\scS)=\supps(k_\scS(1))\neq \Supps(k_\scS(1))=\Supps(f_\scS)$. A more concrete incarnation: Consider the derived category of a rank $1$ non-noetherian valuation domain $(A,\mfm)$, e.g., the perfection of $\mathbb{F}_p[[x]]$, with field of fractions $Q$ and let $P=0$ and $\scS=\mathsf{loc}(Q/\mfm)$. In this case, $\supps(f_\scS)=\{\sloc(\mfm),\sfD_{\{\mfm\}}(A)\}$ and $\Supps(f_\scS)=\{0,\sloc(\mfm),\sfD_{\{\mfm\}}(A)\}=\mathsf{Spc}^\rms(\sfD(A))$; see~\cite[Section 7]{BalchinStevenson21}.
\end{rem}

\subsection{Stratification and descent}
The results that follow are inspired by the article~\cite{ShaulWilliamson21}, in which appear descent theorems about tt-functors between $R$-linear big tt-categories within the context of stratification in the sense of~\cite{BensonIyengarKrause11a} (with subsequent applications in the theory of $DG$-rings). Contrasted with our setup, there are two vital differences. First, the local-to-global principle is a property that holds automatically in their setting. Second, the categories involved have the same spectrum, namely $\Spec(R)$, by assumption. We step closer to the spirit of the alluded configuration by requiring the induced map on smashing spectra to be a homeomorphism.

Let $F\colon \scT\to \scU$ be a coproduct-preserving tt-functor between big tt-categories whose smashing spectra are $T_D$ and assume that $f\coloneqq\Spcsa{F}\colon \Spcsa{\scU}\to \Spcs$ is a homeomorphism. By Stone duality, the map $\sfS(\scT)\to \sfS(\scU)$ that carries a smashing ideal $\scS$ to $\loc(F\scS)$ is a lattice isomorphism. It follows that $f^{-1}(P)=\loc(FP)\in \Spcsa{\scU},\ \forall P\in \Spcs$. Therefore, if $\{P\}=U_\scS\cap V_P$, then $\{f^{-1}(P)\}=U_{\loc(F\scS)}\cap V_{\loc(FP)}$. Hence, the Rickard idempotent corresponding to $\loc(FP)$ is $F(e_\scS)\ot F(f_P)=F(\gG_P)$. Since every smashing-prime of $\scU$ is realized as $\loc(FP)$, for a unique $P\in \Spcs$, we see that the Rickard idempotents of $\scU$ are precisely the images, under $F$, of the Rickard idempotents of $\scT$.
\begin{lem}
\label{lem:auxilliary}%
Let $H\colon \scC_1 \to \scC_2$ be a coproduct-preserving triangulated functor (e.g., $H=X\ot - \colon \scT \to \scT$ for a big tt-category $\scT$, or $H$ could be the right adjoint of a coproduct-preserving tt-functor between big tt-categories). Let $A$ be an object of $\scC_1$. Then, for all $B\in \sloc(A)$, it holds that $H(B)\in \sloc(H(A))$.
\end{lem}

\begin{proof}
Identical to the proof of~\Cref{lem:image-of-ideal}; replace ``localizing ideal" with ``localizing subcategory".
\end{proof}

\begin{lem}
\label{lem:locAX}%
Let $A$ be an object of $\scT$ such that $\mathsf{loc}(A)=\scT$. Then $\mathsf{loc}(A\ot X)$ is a tensor-ideal and $\loc(X)=\mathsf{loc}(A\ot X)$, for all $X\in \scT$.
\end{lem}

\begin{proof}
Let $Z\in \mathsf{loc}(A\ot X)$ and $\scY=\set{Y\in \scT}{Y\ot Z\in \mathsf{loc}(A\ot X)}$. Then $\scY$ is a localizing subcategory of $\scT$ and we claim that $A\in \scY$, i.e., $A\ot Z\in \mathsf{loc}(A\ot X)$. This is deduced by the following two series of implications, where we invoke~\Cref{lem:auxilliary} for the tensor product:
\begin{align*}
A\ot A\in \mathsf{loc}(A)&\Rightarrow A\ot A\ot X\in \mathsf{loc}(A\ot X)\Rightarrow \mathsf{loc}(A\ot A\ot X)\subseteq \mathsf{loc}(A\ot X),
\\
Z\in \mathsf{loc}(A\ot X)&\Rightarrow A\ot Z\in \mathsf{loc}(A\ot A\ot X) \Rightarrow A\ot Z \in \mathsf{loc}(A\ot X).
\end{align*}
This shows that $\scY=\scT$, proving the first part of the statement. For the second part, since $\mathsf{loc}(A\ot X)$ is a tensor-ideal, $\mathsf{loc}(A\ot X)=\loc(A\ot X)\subseteq \loc(X)$. Finally, $1\in \loc(A)$ implies $X\in \loc(A\ot X)$. We infer that $\loc(X)=\mathsf{loc}(A\ot X)$.
\end{proof}

\begin{rem}
\Cref{lem:locAX} can be generalized: If $\mcA$ is a set of objects of $\scT$ such that $\sloc(\mcA)=\scT$, then $\sloc\parens{A\ot X}{A\in \mcA}\in \Loc(\scT)$ and $\sloc(X)=\sloc\parens{A\ot X}{A\in \mcA}$, for all $X\in \scT$.
\end{rem}

In the proof of the following theorem,~\Cref{lem:image-of-ideal} and~\Cref{lem:locAX} will be used without explicit reference.
\begin{thm}
\label{thm:strat-descent}%
Let $F\colon \scT \to \scU$ be a coproduct-preserving tt-functor between big tt-categories whose smashing spectra are $T_D$ and let $G$ be the right adjoint to $F$. Assume that $\Spcsa{F}\colon \Spcsa{\scU}\to \Spcsa{\scT}$ is a homeomorphism. Then:
\begin{enumerate}[\rm(a)]
\item
If $\scT$ satisfies the local-to-global principle, then $\scU$ satisfies the local-to-global principle.
\item
Suppose that there exists a collection of objects $\scX\subseteq \scU$ such that $\mathsf{loc}(\scX)=\scU$ and $\mathsf{loc}(G(\scX))=\scT$. Then: if $\scU$ satisfies minimality, then $\scT$ satisfies minimality.
\item
Suppose that $\mathsf{loc}(1_\scU)=\scU$ and $\mathsf{loc}(G(1_\scU))=\scT$. Then: if $\scU$ satisfies the local-to-global principle, then $\scT$ satisfies the local-to-global principle.
\end{enumerate}
\end{thm}

\begin{proof}
$\phantom{}$
\begin{enumerate}[\rm(a)]
\item
If $\scT$ satisfies the local-to-global principle, then $1_\scT\in \loc\paren{\gG_P}{P\in \Spcs}$. Thus, $1_\scU=F1_\scT\in \loc\paren{F(\gG_P)}{P\in \Spcs}=\loc\parens{\gG_Q}{Q\in \Spcsa{\scU}}$ and the conclusion follows.
\item
Let $X\in \scT$ and assume that $FX=0$. Since $\sloc(G(\scX))=\scT$, it holds that $\loc(X)=\sloc(G(\scX)\ot X)=\sloc(G(\scX \ot FX))=0$. Thus, $X=0$, proving that $F$ is conservative. Now consider a non-zero object $X\in \loc(\gG_P)$. Then the object $FX \in\loc(F(\gG_P))$ must also be non-zero. Therefore, $\mathsf{loc}(\scX \ot FX)=\loc(FX)=\loc(F(\gG_P))=\mathsf{loc}(\scX\ot F(\gG_P))$, with the second equality by minimality of $\scU$. As a result, $\loc(X)=\mathsf{loc}(G(\scX)\ot X))=\mathsf{loc}(G(\scX \ot FX))=\mathsf{loc}(G(\scX\ot F(\gG_P)))=\loc(\gG_P)$. Consequently, $\loc(\gG_P)$ is minimal.
\item
By assumption, $\scU$ satisfies the local-to-global principle and $\sloc(1_\scU)=\scU$. So, every localizing subcategory of $\scU$ is an ideal and $1_\scU\in \loc\parens{\gG_Q}{Q\in \Spcsa{\scU}}$. So, $G(1_\scU)\in \sloc\paren{GF(\gG_P)}{P\in \Spcs}=\sloc\paren{G(1_\scU)\ot \gG_P}{P\in \Spcs}=\loc\paren{\gG_P}{P\in \Spcs}$. Since $G(1_\scU)$ generates $\scT$, the proof is complete.\qedhere
\end{enumerate}
\end{proof}

\begin{cor}
\label{cor:descent-strat}%
Let $F\colon \scT \to \scU$ be a coproduct-preserving tt-functor between big tt-categories whose smashing spectra are $T_D$ and let $G$ be the right adjoint to $F$. Assume that $\Spcsa{F}\colon \Spcsa{\scU}\to \Spcsa{\scT}$ is a homeomorphism.~Provided that $\sloc(1_\scU)=\scU$ and $\sloc(G(1_\scU))=\scT$, if $\scU$ is stratified by the small smashing support, then $\scT$ is stratified by the small smashing support.
\end{cor}

\end{document}